\documentclass[11pt]{article}
\usepackage{geometry}

\usepackage{amsmath}
\usepackage{amssymb,amsthm}
\usepackage{hyperref}
\usepackage{enumerate}
\usepackage{ulem}
\usepackage[arrow, matrix, curve]{xy}

\newcommand{\eni}{\begin{equation}}
\newcommand{\enf}{\end{equation}}
\newcommand{\be}{\begin{equation}}
\newcommand{\ee}{\end{equation}}

\newcommand{\br}{\begin{rem}}
\newcommand{\er}{\end{rem}}
\newcommand{\bex}{\begin{ex}}
\newcommand{\eex}{\end{ex}}

\newcommand{\bc}{\begin{center}}
\newcommand{\ec}{\end{center}}
\newcommand{\bn}{\begin{enumerate}}
\newcommand{\en}{\end{enumerate}}
\newcommand{\bi}{\begin{itemize}}
\newcommand{\ei}{\end{itemize}}
\newcommand{\mfi}{\begin{eqnarray*}}
\newcommand{\mff}{\end{eqnarray*}}
\newcommand{\mfni}{\begin{eqnarray}}
\newcommand{\mfnf}{\end{eqnarray}}

\newcommand{\bt}{\begin{thm}}
\newcommand{\et}{\end{thm}}
\newcommand{\bl}{\begin{lem}}
\newcommand{\el}{\end{lem}}
\newcommand{\bd}{\begin{defi}}
\newcommand{\ed}{\end{defi}}
\newcommand{\bp}{\begin{proof}}
\newcommand{\ep}{\end{proof}}
\newcommand{\bq}{\begin{que}}
\newcommand{\eq}{\end{que}}

\newtheorem{thm}{Theorem}
\newtheorem{defi}{Definition}
\newtheorem{lem}{Lemma}
\newtheorem{prop}{Proposition}

\newtheorem{rem}{Remark}
\newtheorem{ex}{Example}
\newtheorem{ass}{Assumption}
\newtheorem{que}{Question}

\providecommand{\nor}[1]{\left\lVert {#1} \right\rVert}

\providecommand{\abs}[1]{\lvert{#1}\rvert}
\providecommand{\set}[1]{\left\{#1\right\}}

\providecommand{\supp}[1]{\text{supp}( {#1} )}

\providecommand{\scal}[2]{\left\langle{#1},{#2}\right\rangle}

\newcommand{\R}{\mathbb R}
\newcommand{\C}{\mathbb C}

\newcommand{\N}{\mathbb N}

\newcommand{\hh}{\mathcal H}

\newcommand{\la}{\lambda}

\newcommand{\kk}{k}

\include{macros_nicole}

\numberwithin{equation}{section}

\title{Reproducing kernel Hilbert spaces on  manifolds: \\ Sobolev and Diffusion spaces}

\author{Ernesto De Vito\footnote{DIMA, Universita' degli Studi di Genova,,
    {\it devito@dima.unige.it} } \;,  
 Nicole M\"ucke\footnote{Institute for Stochastics and Applications, University of Stuttgart, {\it nicole.muecke@mathematik.uni-stuttgart.de} } \;, 
and \;Lorenzo Rosasco\footnote{LCSL, Universita' degli Studi di Genova, Massachusetts Institute of Technology \& Istituto Italiano di Tecnologia,  {\it lrosasco@mit.edu}} }

\begin{document}
\maketitle

\begin{abstract}
We study  reproducing kernel Hilbert spaces
 (RKHS) on a Riemannian manifold. In particular, we discuss
under which condition Sobolev spaces are RKHS and characterize their reproducing kernels.  Further, 
we introduce and discuss a class of smoother  RKHS that we call diffusion spaces. We illustrate the general results with a number of detailed examples. 
\end{abstract}




\section{Introduction}

Among different notions of function spaces, reproducing kernel Hilbert
spaces (RKHS) play a central role in a number of diverse contexts,
including stochastic analysis \cite{Bog15}- where they are also
known as Cameron-Martin spaces \cite{bog98}, harmonic analysis
\cite{berg84}, \cite{Dau92}, physics \cite{Ali12}, numerical analysis
\cite{Wen10}- where they are also known as native spaces, statistics
\cite{Ber_book}, and machine learning {\cite{cucsma02,stechr2008},} to name a few.
RKHS are Hilbert spaces of functions with continuous evaluation
functionals, a property that naturally yields a number of implications
and characterizations, where positive kernels and corresponding
integral operators are key objects.  Among other references { \cite{aro50}
is a classic.}  Examples of RKHS and kernels abound and include functions defined in Euclidean spaces \cite{Ber_book} but
also for functions on less structured space, for example discrete
space \cite{Sha04}. In many modern applications it is relevant to
consider functions depending on a large, if not huge, number of
variables potentially related to each others. Considering functions
defined on manifolds provide a natural way to formalize this idea. The
goal of this paper is to describe in a self contained manner a number of
examples of RKHS on  Riemannian
  manifolds with bounded geometry \cite{CGT,trib92}. As we show,
if the smoothness index is large enough,  Sobolev spaces 
provide a primary example of RKHS.  We observe that in the
  literature there are many different definitions of Sobolev spaces
  and the technical assumption that the manifold has bounded geometry,
  see item~\ref{def_bounded}) of Proposition~\ref{prop_M}, is needed
  to ensure that the various approaches are equivalent to each other.
  Examples of manifolds of bounded geometry are: compact Riemannian
  manifolds and Lie groups with an invariant Riemannian
  structure. In the paper,  after collecting in a
unified way a number of definitions and results on Sobolev spaces, we
show under which condition they are RKHS and characterize the
corresponding kernels and integral operators using spectral
theory. Further, we introduce a class of functions spaces, called
diffusion spaces, defined by the heat kernel which naturally
generalize the RKHS with Gaussian kernels in a Euclidean setting.
Finally, we illustrate the general discussion presenting a number of detailed
examples.

While connections between Sobolev spaces, differential operators and
RKHS are well known in the Euclidean setting, here we present a self
contained { study} of analogous connections for Riemannian manifolds.  By
collecting a number of results in unified a way we think our study can
be useful for researchers interested in the topic. 

 The rest of the
paper is organized as follows. In Section~\ref{sec:notation} we set
the notation and introduce basic concepts and assumptions. In
Section~\ref{sec:sobolev} we recall different notions and results on
Sobolev spaces of functions on a Riemannian manifold. In
Section~\ref{sec:diffusion} we introduce the concept of diffusion
spaces. In Section~\ref{sec:compact} we specialized the previous
definitions and results to the case of compact manifolds where a
number of simplifications occur. Finally, in Section~\ref{sec:RKHS} we
provide an RKHS perspective on the function spaces previously
introduce and illustrate them with a number of examples in
Section~\ref{sec:examples}.

\section{Notation and assumptions}\label{sec:notation}

In this section we fix the notation and state the main assumptions. We
refer to Appendix~\ref{review} for definitions and results on
Riemannian geometry.  In this paper, we consider the class
of Riemannian manifolds satisfying the following assumption.
\begin{ass}\label{ass_M}
Let $M$ be an $n$-dimensional manifold, which is
connected, complete and  with bounded geometry.
\end{ass}
The manifold $M$ has {\em bounded geometry} if the estimates~$\eqref{eq:18}$ and
$\eqref{eq:27}$ given in the Appendix hold true.  
We denote by $g$ and $\nabla$ the Riemannian metric and the
corresponding Riemannian connection, respectively.  The Riemannian
metric $g$ induces a distance on $M$ and, by Assumption  $M$ becomes a complete
metric space, see item~\ref{distance}) of Prop.~\ref{prop_M}. We
denote by $B(m,r)$ the ball of center $m\in M$ and radius $r>0$. 

In many examples, $M$ is an embedded submanifold of $\R^d$ with the
induced Riemannian structure. In Appendix~\ref{review} we recall
some properties and we provide some explicit formulae for $g$ and $\nabla$.

Some typical examples are:

\begin{example}
The space $\R^n$ with the usual
Riemannian structure induced by the 
Euclidean scalar product satisfies
Assumption~\ref{ass_M}. 
\end{example}

\begin{example}
Any compact connected submanifold of $\R^d$ satisfies
Assumption~\ref{ass_M}. Indeed, the Hopf-Rinow theorem implies that $M$
is complete, see item~\ref{Hopf}) of Prop.\ref{prop_M}, and $M$ has
bounded geometry by the Weierstrass theorem.   
\end{example}

\paragraph{Normal coordinates.} In order to introduce the Sobolev
spaces, one needs a nice family of local charts on $M$,
whose existence is ensured by the
following result.

\begin{thm}\label{partition}
Given $r>0$ small enough, there exists a smooth atlas $\set{U_j,\varphi_j}_{j\in J}$ of $M$ such that
for all $j\in J$
\begin{equation}
    \label{eq:2}
    U_j=B(m_j,r)\subset M \qquad \varphi_j:U_j\to\R^n\quad
    \varphi_j(m)=\exp_{m_j}^{-1}(m), 
  \end{equation}
where $\set{m_j}_{j\in J}$ is  a suitable family of points in $M$. 
Furthermore, there exists a family $\set{\psi_j}_{j\in J}$ of smooth real
functions on $M$ such that
\begin{equation}
  \label{eq:3}
 0\leq \psi_j\leq 1 \qquad \operatorname{supp}{\psi_j
      }\subset U_j \qquad \sum_{j\in J} \psi_j = 1 .
\end{equation}
\end{thm}
We add some comments to explain the statement. Denoted by
  $\operatorname{inj}(m)$  the injectivity radius at $m$,  see
  item~\ref{inj_radious}) of Prop.~\ref{prop_M}, and 
  \[r_M  =\inf_{m\in M} \operatorname{inj}(m),\]
then by~\eqref{eq:18} $r_M>0$ and for any $r<r_M$
Theorem~\ref{partition} holds true.
In~\eqref{eq:2}, the map $\exp_m:T_m(M)\to M$ denotes the exponential
map at $m\in M$. By choosing an orthonormal base,  $T_m(M)$ is
identified with $\R^n$ and, by item~\ref{exp}) of Prop.~\ref{prop_M}),
$\exp_m$ is a diffeomorphism from $B(0,r)\subset\R^n$  onto
$B(m,r)\subset M$. The inverse $\varphi_m=\exp_{m}^{-1}: B(m,r)\to\R^n
$ is called normal coordinates at $m$ since they satisfy~\eqref{eq:30}.  
By definition of an atlas, 
the family $\set{U_j}_j$ is a locally finite open covering on $M$,
Eq.~\eqref{eq:3} states that $\set{\psi_j}_{j\in J}$ is a smooth partition of
unity subordinate to the open covering    $\set{U_j}_{j\in J}$ and
\[\operatorname{supp}{\psi}=\overline{\{m\in M\mid \psi(m)\neq 0\}}\]
 denotes the support of the continuous function $\psi$. 
By our assumption on $M$, the index set $J$ might be chosen countable and we take it finite if 
$M$ is compact.

\paragraph{The volume measure.} The metric $g$ induces a Radon measure
on $M$, which plays the role of the Lebesgue measure of
$\R^d$. Indeed, there exists a unique Radon measure $dm$ on $M$,
called the Riemannian volume measure \cite{sa96}[Chap.1~\S~5.1]
  or \cite{cha84}[Ch. 3, \S~3] such that
\begin{equation}
  \label{eq:29}
 \int_M f(m)  dm = \sum_{j\in J} \int\limits_{ B(0,r)}  \psi( \varphi^{-1}_j(x))\, f(\varphi^{-1}_j(x))
 \sqrt{\det{g(x)}}\, dx,
\end{equation}
where $dx$ is the Lebesgue measure of
$\R^n$ and  $\det{g}$ is the determinant of the metric $g$ in local
coordinates (see item~\ref{exp}) of Prop.~\ref{prop_M}). If $M$ is orientable, it is
  possible to define  a volume form $d\Omega$ such that, if the
  ortonormal normal base of $T_m(M)$ is positive oriented, then
  $d\Omega=\sqrt{\det{g(x)}}dx^1\wedge\ldots\wedge dx^n$,
see~\cite[page
57]{pet16}. 

Given $p\in [1,+\infty)$, we denote by $L^p(M)$ 
the Banach space of (equivalence classes of) $p$-integrable real
functions on $M$ with the corresponding norm $\nor{\cdot}_p$ and, for
$p=2$,  $\scal{\cdot}{\cdot}_2$ is the corresponding scalar product.

\paragraph{The Laplacian.}  The Riemannian connection $\nabla$
defines the Laplacian on the space of smooth functions as 
\[
\Delta f(m)= -\sum_{i=1}^n g( \nabla_{e_i} \nabla f , e_i),
\]
where $\nabla f$ is the unique vector field such that
\[ g( \nabla f , X) =X(f),\]
and $\set{e_i}_{i=1}^n$ is any orthonormal base of $T_m(M)$. In local
coordinates, see \cite[page 57]{pet16},
\begin{equation}
  \label{eq:28}
  \Delta f = - \frac{1}{\sqrt{\det{g}}}
\partial_i \left(  g^{ij} \sqrt{\det{g}} \  \partial_j f\right),
\end{equation}
where $\det{g}$ and $g^{ij}$ are defined in item~\ref{exp}) of
Prop.~\ref{prop_M}. 
We use the Einstein sum convention.

\begin{rem}
We observe that, if the Riemannian metric is modified  by a conformal
 change, the Riemannian volume measure is multiplied by  a smooth
 nowhere vanishing density and this change reflects to the form of the
Laplacian. More explicitly,  if 
\[ \tilde g = \rho^2 g \;, \qquad \rho \in \cC^\infty(M),\,  \rho(m)>0\] 
denotes the conformally equivalent metric, then one obtains that the
associated Riemannian volume measure $d \tilde m = \rho dm$  
and the Laplacian associated to $\tilde \rho$ is given in local coordinates by 
\begin{align}\label{eq:lap}
\tilde \Delta  &=  \frac{1}{\rho^2} \Delta + \frac{1}{\rho \sqrt g} \sum \brac{ \partial_i ,\rho^{-1} }\sqrt g \; g^{ij} \partial_j  \\ 
&= \rho^{-2} \Delta - \rho^{-3} \sum_{ij} (\partial_i \rho) g^{ij} \partial_j  \;, \nonumber 
\end{align}
In~\eqref{eq:lap}, we denote by $\brac{\cdot, \cdot}$  the commutator $\brac{A,B} = AB-BA$.
\end{rem}

The sign convention is such that $\Delta$ is a positive operator
on $L^2(M)$ as stated by the following result,
see~\cite[Thm.~2.4]{str83}.  We denote by ${\cal D}(M)$ the
  space of smooth functions on $M$ with compact support, which is a
  subspace of $L^2(M)$ since compacts sets have finite measure.

\begin{thm}  \label{esssad}
The operator $\Delta: {\cal D}(M)\to L^2(M)$  uniquely extends to a self-adjoint
    unbounded operator on $L^2(M)$ and this extension, denoted again by
    $\Delta$, is a positive operator.
\end{thm}
\begin{remark}
The assumption that $M$ is complete is crucial for the uniqueness 
  statement of Theorem \ref{esssad}, i.e. to ensure that 
  $\Delta:{\cal D}(M) \to L^2(M) $ is {\em essentially} self-adjoint.
If $M$ is an arbitrary Riemannian manifold, since  $\Delta:{\cal D}(M) \to L^2(M) $ is a
symmetric positive operator, Friedrich's extension theorem  \cite{rs}
always provides a self adjoint extension $\Delta_F$, but for
incomplete manifolds there are many self-adjoint extensions,
corresponding to different boundary conditions, and none of the
  equivalence statements in the definition of Sobolev spaces given
  below in Theorem \ref{sobolev} survives in this case, see
  e.g. \cite{e}. This is one of the main reasons why we stick to
  Assumption \ref{ass_M}.
\end{remark}

Given a Borel function $\Phi:[0,+\infty)\to \R$,  the spectral calculus
allows to define an (unbounded) operator  $\Phi(\Delta)$  acting on
$L^2(M)$ as
\begin{equation}
  \label{eq:4}
  \scal{\Phi(\Delta) f}{g}_2= \int_0^{+\infty} \Phi(\la)
  dP_{f,g}(\la) \;, \qquad f\in\operatorname{dom}\Phi(\Delta),\; g\in L^2(M)
\end{equation}
with domain
\[
\operatorname{dom}\Phi(\Delta)=\set{f\in L^2(M)\mid \int_0^{+\infty} \Phi(\la)^2
  dP_{f,f}(\la)}.
\]
Here,  for all Borel subsets $E\subset[0,+\infty)$, $E\mapsto P(E)$ is the spectral
measure associated with $\Delta$, and $dP_{f,g}(\la)$ denotes integration w.r.
to the complex measure
$P_{f,g}(E)= \scal{P(E) f}{g}_2 $, see  \cite[Chapter XX]{lang93}.


\section{Sobolev spaces}\label{sec:sobolev}
A canonical way to define function spaces that encode the geometry of
the underlying manifold is through the notion of Sobolev spaces. In
the literature there are different  approaches. Here we
collect all the equivalent definitions.  Since we are interested in
Hilbert spaces, we state the result for $p=2$, however they hold true
for any power $p\in [1,+\infty)$ with minor modifications. 
 We denote by $\mathcal D'(M)$ the
space of distributions on $M$.  Furthermore,
$\set{U_j,\varphi_j}$ and $\set{\psi_j}_{j\in J}$ are  the atlas and
  the partion unity  given by   Prop.~\ref{partition}.
\begin{subequations}
  \begin{thm}   \label{sobolev}
    Fix $s\in [0,+\infty)$ and let $M$ satisfy Assumption \ref{ass_M}. Then, for any distribution $f\in \mathcal D'(M)$, the following conditions
    are equivalent.
    \begin{enumerate}[a)]
    \item 
      \begin{equation}
        \label{eq:35}
   \nor{f}_{H^s,1}^2 =\sum_j \nor{ \psi_j  f \circ\varphi^{-1}_j }^2_{H^s(\R^n)} <+\infty,
      \end{equation}
where $(\psi_j f) \circ\varphi^{-1}_j$ is regarded as tempered distribution on
$\R^d$, which is zero outside  the ball $B(0,r)$.
\item There exists $g\in L^2(M)$ such that
  \begin{equation}
    \label{eq:36}
       f= (\Id+\Delta)^{-\frac{s}{2}} g \qquad  \nor{f}_{H^s,2} =\nor{g}_2,
  \end{equation}
where $(\Id+\Delta)^{-\frac{s}{2}}$ is the Bessel potential  associated
with the function $\Phi(\la)=(1+\la)^{-s/2}$  by spectral calculus. 
\item The distribution $f$ is in the domain of 
  $\Delta^{s/2}$ and
  \begin{equation}
    \label{eq:37}
    \nor{f}_{H^s,3}^2 =\nor{f}^2_2+ \nor{\Delta^{s/2}f}_2^2,
  \end{equation}
where $\Delta^{s/2}$  is the Riesz potential associated
with the function $\Phi(\la)=\la^{s/2}$  by spectral calculus. 
    \end{enumerate}
If one of the above conditions is satisfied, there exists constants $c_1,c_2,c_3,c_4>0$,
independent of $f$, such that
\[
 c_1  \nor{f}_{H^s,1} \leq c_2\nor{f}_{H^s,2} \leq c_3 \nor{f}_{H^s,3}
 \leq c_4 \nor{f}_{H^s,1} .
\]
If $s\in \N$, then $H^s(M)$ is the completion of the space
\begin{equation}
  \label{eq:38}
  \set{ f\in C^\infty(M) \mid \nor{f}_{H^s,4}^2=\sum_{\ell=0}^s
    \nor{\nabla^\ell f}^2_2 <+\infty},
\end{equation}
with respect to the norm $\nor{\cdot}_{H^s,4}$, which is equivalent to
$\nor{\cdot}_{H^s,1}$, $\nor{\cdot}_{H^s,2}$, $\nor{\cdot}_{H^s,3}$.
\end{thm}
\end{subequations}
In~\eqref{eq:38}  $\nabla^\ell$ denotes the $\ell$-fold 
composition of the Riemannian connection $\nabla$ considered as map 
 from $ {\cal A}^0(TM)$ to ${\cal A}^1(TM)$, where ${\cal A}^j(TM)$
 for $j \in \N$ denotes the module of $TM$- valued $j-$forms on $M$;
 in particular, for $j=0,1$ as above, and $X,Y$ being smooth vector
 fields on $M$, i.e. sections of $TM$, the contraction of 
$\nabla Y$ with $X$ is denoted by $\nabla_X Y$ and is thought of as the derivative 
of $Y$ in the direction of $X$ (or $X_m$, since the connection is
tensorial w.r. to $X$). By an habitual abuse of notation, connections
might then be composed, yielding a map $\nabla^j:{\cal A}^0(TM) \to
{\cal A}^j(TM)$.

\begin{proof}
Definition~\eqref{eq:35} is given in \cite[page 286]{trib92} where it is
denoted as $F^s_{22}(M)$  (note that $F^s_{22}(\R^n)=H^s(\R^n)$ as shown in
\cite[page 18 and 1.3.3 Eq.~(13)]{trib92}) and $F^s_{22}(M)$ also coincides with
the Besov space $B^s_{22}(M)$, see \cite[Thm. 3.7.1, page
309]{trib92}. Definition~\eqref{eq:36} is given in 
\cite[Def.~4.1]{str83} and the equivalence with
Definition~\eqref{eq:37} is shown in \cite[Thm~4.4]{str83}.
Definition~\eqref{eq:38} is given in \cite[Def.~2.3]{aub98}.  The
equivalence of Definition~\eqref{eq:35} with Definition~\eqref{eq:36}
and Definition~\eqref{eq:38} is given in \cite[page 320]{trib92} (see
\cite[Definition page 319 and Remark 1.4.5/1 page 301]{trib92} for the choice $\rho=1$). 
\end{proof}
\begin{rem}
Since $M$ has bounded geometry, {\rm Thm.}~$3.1$ and {\rm Prop.}~$3.2$ in
{\rm \cite[page 49]{hebey2000nonlinear}} show that in~\eqref{eq:30}
$C^\infty(M)$ can be replaced by $\mathcal D(M)$.
\end{rem}

Based on the above theorem,  we are able to define the Sobolev space
$H^s(M)$.
\begin{defi}
Given $s\in [0,+\infty)$,  let  $H^s(M)$ be the set of distributions
$f\in \mathcal D'(M)$ satisfying    one of the equivalent
conditions~\eqref{eq:35}, \eqref{eq:36} or \eqref{eq:37}. 
The space $H^s(M)$ becomes a Hilbert space with respect to one of
  the  bilinear forms
  \begin{subequations}
    \begin{alignat}{1}
      \scal{f}{g}_{1,s} & = \sum_j \scal{ (\psi_j f)
        \circ\varphi^{-1}_j }{ (\psi_j g)
        \circ\varphi^{-1}_j }_{H^s(\R^n)}      \label{eq:6a} \\
      \scal{f}{g}_{2,s} & =\scal{ (\Id+\Delta)^{\frac{s}{2}}f }{
        (\Id+\Delta)^{\frac{s}{2}} g}_{2} \label{eq:6b} \\
      \scal{f}{g}_{3,s} & = \scal{f}{g}_{2} +
      \scal{\Delta^{\frac{s}{2}} f}{\Delta^{\frac{s}{2}} g}_2 .
    \end{alignat}
  \end{subequations}
\end{defi}
In general, the above equivalent definitions of Sobolev spaces depend on
the metric $g$, however if $M$ is compact it is possible to show that
$H^s(M)$ is independent on the metric \cite[Prop.~2.2]{aub98}.  

 It is interesting to recall
  the following interpolation
result \cite[Theorem7.4.4]{trib92}.  Given $0\leq s_0<s <s_2$, set $r\in (0,1)$ such that
$s=(1-r)s_0+r s_1$, then
  \begin{equation}
H^s(M)=[H^{s_0}(M),H^{s_1}(M)]_{r,2}\label{eq:16}
\end{equation}
where $[H^{s_0},H^{s_1}]_{r,2}$ denotes the interpolation space given
by the real interpolation method, see~\cite{Ben88}.

The fact that  $H^s(M)$ is a reproducing kernel Hilbert space
provided that the smoothness index $s$ is large enough is based on the
following embedding theorem, which needs some care. 

Recall that, given $\sigma>0$, the  H\"older-Zygmund
space is  defined as, \cite[page 314]{trib92},  
\begin{equation}
  \label{eq:40}
\mathcal{CH}^\sigma=\set{ f\in D'(M)\mid \nor{f}_{\mathcal C^\sigma}
  =\sup_{j\in J} \nor{
  (\psi_j f) \circ \varphi^{-1}_j  }_{\mathcal {CH}^\sigma(\R^n)} <+\infty},
\end{equation}
where the notation is as in~\eqref{eq:35} and $\mathcal
{CH}^\sigma(\R^n)=B^s_{\infty,\infty}$ is the classical H\"older-Zygmund space on
$\R^n$, \cite[Section 1.2.2 and Thm. in Section
1.5.1]{trib92}). Furthermore, we denote by $C(M)$ the space of continuous functions
endowed with the topology of compact convergence. 

We are now ready to state the Sobolev embedding theorems. 
\begin{thm}\label{embedding}
Given $s<s'$ and 
\begin{subequations}
  \begin{equation}
    \label{eq:39}
    H^{s'}(M) \hookrightarrow  H^{s}(M). 
  \end{equation}
If $s>n/2$ and $0<\sigma\leq s-n/2$
\begin{equation}
  \label{eq:9}
 H^{s}(M) \hookrightarrow  \mathcal{CH}^{\sigma}(M)  \hookrightarrow C(M).
\end{equation}
\end{subequations}
\end{thm}

\begin{proof}
The proof can be found in \cite[Thm. page 315, item iii) and
iv)]{trib92} taking into account  that $H^s(M)=F^s_{22}(M)=B^s_{22}(M)$. 
The inclusion  $H^{s'}(M) \hookrightarrow H^{s}(M)$ is also proven in
\cite[Thm. 4.2.]{str83}.  
\end{proof}
\begin{rem}
The assumption that $M$ has bounded geometry implies that the
  Ricci tensor $R$ is bounded from
    below, {\em i.e.} there exists a constant $k\in \R$ such that
      \begin{equation}
R(X,X) \geq k g(X,X) \qquad X\in T(M).\label{eq:bounded}
\end{equation}
{If we only assume that the Ricci tensor is bounded from below
  and, for any $n\in\N$, we} define the Sobolev space $H^n(M)$
by~\eqref{eq:38}, then there is the following embedding result.  For
all $s\in\N$ such that $s>n/2+k$,  then   
  \begin{equation}
H^s(M) \hookrightarrow C^k_b(M),\label{eq:7}
\end{equation}
where $C_b^k(M)$ is the space of  $C^k$-functions with bounded derivatives up to 
order $k$ and the embedding is continuous,  see   {\rm \cite[Thm. 2.9
    and Thm. 3.4]{hebey2000nonlinear}} or  {\rm
    \cite[Prop.~3.3]{tay11}} if $M$ is compact. 
See the discussion in~{\rm \cite[Section 1.2.2 and Section 7.5.3]{trib92}} about the
  difference between  the space $C_b^k(M)$ and H\"older-Zygmund space
  $\mathcal{CH}^k(M)$. Note that condition~\eqref{eq:bounded}  is the standard
  assumption for volume comparison theorems as Bishop's Theorem
  \cite{cha84}[Theorem~3.9] and Gromov's Theorem
  \cite{cha84}[Theorem~3.10].

  If $M$ has bounded geometry,~\eqref{eq:7} and~\eqref{eq:39} imply
  for all $s\in \R$ and $k\in \N$ such that $\lfloor s\rfloor>
  k+n/2$, that 
  \[
H^s(M) \hookrightarrow C^k_b(M).
 \]   
\end{rem}
Finally, if $M$ is compact the following Rellich-Kondrakov theorem
holds true,
\cite[Prop.~3.9]{hebey2000nonlinear}, \cite[Thm.
3.9]{hebey2000nonlinear} and \cite[Thm. 2.34]{aub98}.
\begin{thm}\label{rellich}
Assume that $M$ is compact. For any $0<s<s'$, the embedding
\[   H^{s'}(M) \hookrightarrow H^{s}(M) \]
is compact. Furthermore, if $s>n/2$, the embedding
\[H^{s}(M) \hookrightarrow C(M)=C_b(M)\]
is compact, too. 
\end{thm}


\section{Diffusion spaces}\label{sec:diffusion}

We introduce a class of functions that we call diffusion spaces, inspired by the line of work on diffusion geometry in machine learning and harmonic analysis, see e.g. \cite{CoiLafLee05}. 
The idea is to  encode the geometry of $M$ into smooth function
spaces  by  means of the heat kernel, which plays a role analogous to the 	
Gaussian kernel in  $M=\R^d$. We first review the main properties of
the heat kernel and then we introduce the corresponding diffusion spaces. 

For all $t>0$, denote by ${\mathrm e}^{-t\Delta}$ the heat kernel, defined as
bounded operator on $L^2(M)$ by spectral calculus, see~\eqref{eq:4}
with $\Phi_t(\la)={\mathrm e}^{-t\la}$.   There is the
following result \cite[Thm.s 3.5 and 3.6]{str83}.
\begin{thm}\label{heat_kernel}
There exists a unique smooth function $p:M\times M\times (0,+\infty)\to \R$
such that
\begin{enumerate}
\item for all $m\in M$ and $t>0$, the function $p(m,\cdot,t)\in L^1(M)$
  and $\nor{p(m,\cdot,t)}_{L^1(M)}\leq 1$;
\item for all $m,m'\in M$ and $t>0$ 
  \begin{equation}
    \label{eq:10}
    p(m,m',t)=p(m',m,t)>0;
  \end{equation}
\item for all $f\in L^2(M)$
  \begin{equation}
    \label{eq:43}
    {\mathrm e}^{-t\Delta}f(m)= \int\limits_M p(m,m',t) f(m') \,dm';
  \end{equation}
\item given $f\in L^2(M)$, for all $t>0$
  \begin{equation}
    \label{eq:44}
    \nor{ {\mathrm e}^{-t\Delta}f}_2 \leq \nor{f}_2 \qquad \lim_{t\to 0^+} \nor{{\mathrm e}^{-t\Delta}f-f}_2=0;
  \end{equation}
\item given  $f\in L^2(M)$, for all $t>0$ the function ${\mathrm e}^{-t\Delta}f$ is smooth and 
  \begin{equation}
    \label{eq:45}
    \frac{\partial}{\partial t} {\mathrm e}^{-t\Delta}f = -\Delta {\mathrm e}^{-t\Delta}f .
  \end{equation}
\end{enumerate}
\end{thm}
The fact that ${\mathrm e}^{-t\Delta}$ is a semigroup and the uniqueness of the
kernel implies that 
\begin{equation}
  \label{eq:50}
    p(m,m',t+s) = \int_M p(m,m'',t) p(m'',m',s) \ dm''  \qquad t,s>0,\,
    m,m'\in M.
\end{equation}
We are now ready to define the diffusion spaces. In the literature
there is no a standard notation. 
For all $t>0$, set 
\begin{equation}
  \label{eq:47}
  \hh^t ={\mathrm e}^{-\frac{t}{2}\Delta} L^2(M) ,
\end{equation}
which becomes a Hilbert space with respect to the scalar product
\begin{equation}
  \label{eq:48}
  \scal{f}{h}_{\mathcal H^t} = \scal{{\mathrm e}^{\frac{t}{2}\Delta} f}{{\mathrm e}^{\frac{t}{2}\Delta} g}_2. 
\end{equation}
The following result is a direct consequence of the definition.
\begin{prop}
 For all $0<t<t'$ and $s>0$
\[\hh^{t'}\hookrightarrow \hh^{t} \hookrightarrow  H^{s}(M).\]   
\end{prop}
\begin{proof}
 The semi-group property of ${\mathrm e}^{-t\Delta}$ shows that
\[
  \hh^{t'} = {\mathrm e}^{-\frac{t}{2}\Delta} {\mathrm e}^{-\frac{(t'-t)}{2}\Delta} L^2(M)\subset {\mathrm e}^{-\frac{t}{2}\Delta}
  L^2(M)=  \hh^{t},
\]
and the inclusion is continuos since $ {\mathrm
  e}^{-\frac{t}{2}\Delta} $ is bounded. 

Since the function $\varphi_t(\sigma)={\mathrm e}^{-t \sigma/2}
(1+\sigma)^{s/2}$ is bounded on $[0,+\infty)$, the operator
$\varphi_t(\Delta)$ is bounded on $L^2(M)$. If $f\in \hh^t$, then for
some $g\in L^2(M)$
\[
f = {\mathrm e}^{-\frac{t}{2}\Delta} g =
  (\Id+\Delta)^{-\frac{s}{2}}\varphi_t(\Delta) g =
  (\Id+\Delta)^{-\frac{s}{2}} h
\]
with $h=\varphi_t(\Delta) g  \in L^2(M)$, so that $f\in H^s(M)$, and
\[ \nor{f}_{H^s(M)}=\nor{g}_2\leq \nor{ \varphi_t(\Delta)}_2
  \nor{g}_2=   \nor{ \varphi_t(\Delta)}_2\nor{f}_{\hh^t}.\]
\end{proof}

\section{Compact manifolds}\label{sec:compact}
If $M$ is compact, the above equations are easier to write since
$\Delta$ admits a base of eigenfunctions, as shown  by the following classical result.
\begin{thm}[Sturm-Liouville decomposition]\label{sturm-liou}
Assume that $M$ is compact. There exists an orthonormal base
$\set{f_k}_{k\in\N}$ of $L^2(M)$ such  that  each function $f_k$ is smooth and
  \begin{equation}
    \label{eq:31}
    \Delta f_k = \lambda_k \,f_k\qquad k\in\N ,
  \end{equation}
where
\[ 0=\lambda_0 < \lambda_1\leq \lambda_2\leq \ldots \lambda_k\leq
  \ldots \qquad \lim_{k\to+\infty} \lambda_k=+\infty,\]
and the  multiplicity of each $\lambda_k$ is finite (each eigenvalue is repeated according to its multiplicity). Furthermore, 
there exist two universal constants $C>0$ and $k^*$ such that  for all $k\geq k^*$ 
\begin{equation}
  \label{eq:1}
  \abs{f_k(m)} \leq C \la_k^{n/4}   \qquad m\in M.
\end{equation}
Finally, the vector space $\operatorname{span}\set{f_k\mid
  k\in\N}$ is dense in $C^\ell(M)$ for all $\ell\in\N$
\end{thm}
\begin{proof}
The claims can be found in~\cite[page 139]{cha84} or \cite[page
53]{ber86}, up to the bound~\eqref{eq:1}, which is proved 
in Lemma~3.1 of \cite{por16}. 
\end{proof}
We remark that the estimate~\eqref{eq:1} is only slightly better as a trivial application of the Sobolev embedding theorem and not sharp in many cases.
E.g., if $M=S^1$, the eigenvalues are $\lambda_k=k^2$ (to be counted twice according to their multiplicity for $k \neq 0$), while the eigenfunctions
$f_k(t)=(2 \pi)^{-1/2} \exp(ikt)$, normalized in $L^2(M)$ for $k \in
\mathbb Z$, are uniformly bounded in $L^\infty(M)$, independent of $k$. 

Note that~\eqref{eq:3} simplifies as
\begin{equation}
  \label{eq:5}
  \Phi(\Delta) f = \sum_{k=0}^{+\infty} \Phi(\la_k) \scal{f}{f_k}_2
  f_k \qquad \operatorname{dom}\Phi(\Delta)=\set{f\in L^2(M)\mid
    \sum_k \Phi(\la_k)^2 \,\abs{\scal{f}{f_k}_2}^2},
\end{equation}
where the first series is unconditionally convergent in $L^2(M)$. As a
consequence, given $s\in (0,+\infty)$ and $f\in L^2(M)$
the following facts are equivalent
  \begin{alignat*}{2}
   &  f\in H^s(M) \quad &\Longleftrightarrow & \quad
    \nor{f}^2_{H^s,2}=\sum_{k} (1+\lambda_k)^{s} |\scal{f}{f_k}|^2
    <+\infty \\
    &   \quad & \Longleftrightarrow & 
    \quad
    \nor{f}^2_{H^s,3}=\sum_{k} (1+\lambda_k^{s}) |\scal{f}{f_k}|^2
    <+\infty.
  \end{alignat*}
Finally, for any $t>0$,
\begin{equation}
  \label{eq:46}
  p(m,m',t)= \sum_k {\mathrm e}^{-t\lambda_k} f_k(m) f_k(m'),
\end{equation}
where the series converges absolutely and uniformly on $M$, \cite[page
139]{cha84}. Furthermore, given $f\in L^2(M)$ 
\begin{equation}
  \label{eq:49}
 f\in \hh^t \qquad \Longleftrightarrow \qquad   \nor{f}^2_{\mathcal H^t} = \sum_k {\mathrm e}^{t\lambda_k} |\scal{f}{f_k}|^2<+\infty.
\end{equation}


\section{Reproducing kernel Hilbert spaces}\label{sec:RKHS}

In this section, we show that the Laplacian and the heat kernel allow
to define a class of reproducing kernel Hilbert spaces on the
manifold $M$. We refer to Appendix~\ref{app:rkhs} for basic definitions on RKHS. 

By construction,  $H^s(M)$ is continuously embedded in $L^2(M)$ and we
denote by $J_s$ the inclusion. 

\begin{thm} \label{rkhs}
Let $M$ be a manifold satisfying Assumption~\ref{ass_M}.
  \begin{enumerate}[i)]
  \item For any $s\in (0,+\infty)$ such that $s > n/2$, the Sobolev
    space $H^s(M)$ is a reproducing kernel  Hilbert space on $M$,
    its reproducing kernel $K_s$ is separately continuous and locally
    bounded and
    \begin{equation}
J_sJ_s^*  = (I+\Delta)^{-s} = L_{K_s}\label{eq:11} .
\end{equation}
where $L_{K_s}:L^2(M)\to L^2(M)$ is the integral operator with kernel $K_s$.
\item If $M$ is compact,
    then the kernel is jointly continuous and bounded.
\item For all $t>0$ the space $\hh^t$ is a reproducing
    kernel Hilbert space whose reproducing kernel is
    $p(t,\cdot,\cdot)$.
  \end{enumerate}
\end{thm}
\begin{proof}
The first claim is a consequence of~\eqref{eq:9}. Indeed, given a
compact subset $A\subset M$, 
\[ \abs{f(m)}\leq \sup_{m\in A} \abs{f(m)}\leq C_A \nor{f}_{H^s(M)}
  \qquad m\in A, f\in H^m(M),\]
where $C_A$ is a constant independent of $m$ and $f$. Hence, by Riesz
lemma, there exists $K_m\in H^s(M)$ such that 
\[ f(m)= \scal{f}{K_m}_{H^s(M)}\]
with $ \nor{K_m}_{H^s(M)}\leq C_A$, so that the kernel
$K_s(m,m')=\scal{K_m}{K_{m'}}_{H^s(M)}$ is locally bounded. Since
$K_m\in C(M)$, then $K_s$ is separately continuous.
To show~\eqref{eq:11}, take $F\in L^2(M)$ and $f\in H^s(M)$, then 
\begin{alignat*}{1}
  \scal{F}{J_sf}_2 &= \scal{(\Id+\Delta)^s(\Id+\Delta)^{-s}F }{ J_sf }_2 =
  \scal{ (\Id+\Delta)^{s/2} (\Id+\Delta)^{-s}F } { (\Id+\Delta)^{s/2} J_sf }_2 \\
\end{alignat*}
where both $ (\Id+\Delta)^{-s}F $ and $f$ are in $\operatorname{dom}(\Id+\Delta)^{s/2}$. Then there exists $g\in
H^s(M)$ such that $J_sg=(\Id+\Delta)^{-s}F$ and 
\[ \scal{F}{J_sf}_2=\scal{g}{f}_{H^s(M)}.\]
Since $f$ is arbitrary, it follows that $J_s^*F=g$, so that $J_sJ_s^* F=
(\Id+\Delta)^{-s}F$. On the other hand, since $H^s(M)\subset L^2(M)$, 
$J_sJ_s^*$ is the integral operator with kernel $K_s$ \cite[Prop.~4.4]{carmeli2006vector}.

We now prove item~ii). By a general result on reproducing kernel Hilbert spaces, the kernel 
$K_s$   is jointly continuous if and only if the map
$m\mapsto K_m$ is continuous from $M$ to $H^s(M)$. Denoted by $B_1$ the
unit ball in $H^s(M)$, Since
\[ \nor{K_m - K_{m_0}}_{H^s(M)} =\sup_{f\in B_1} \abs{f(m)-f(m_0)},\]
the joint continuity is equivalent to the fact that  the family 
$B_1$ regarded as a subset of $C(M)$ is equicontinuous. {Since $\set{f(m)\mid f\in B_1}\subset \R$ is bounded by $\nor{K_m}$, Ascoli-Arzel\'a
theorem, which holds true for any locally compact space,}  implies that this last condition is 
equivalent to the fact that the embedding of $H^s(M)$ into $C(M)$ is compact, see
\cite[Prop. 5.3]{carmeli2006vector}.   {Since } $M$ is compact,
Thm.~\ref{rellich} provides the conclusion. 

{We now prove iii). Fix $t>0$. For any $m\in M$, by}~\eqref{eq:50} (with the choice $s=t$ and $m'=m$) and the symmetry
of the kernel, it follows that $p(m,\cdot,t)\in L^2(M)$. Given
$f\in \hh^t$, $f= {\mathrm e}^{-\frac{t}{2} \Delta} g$  with $g\in L^2(M)$, for all  $m\in M$~\eqref{eq:43} gives 
\[
f(m) = \int_M p\big(m,m',\frac{t}{2}\big) g(m') \, dm'  = \scal{p\big(m,\cdot,\frac{t}{2}\big)}{g}_2.
\]
By Cauchy-Schwarz inequality
\[
\abs{f(m)} \leq \nor{p(m,\cdot, \frac{t}{2})}_2 \nor{g}_2 = C_m \nor{f}_{\hh^t}, 
\]
so the evaluation functional at $m$ is continuous. Furthermore, with
the choice
\[
K_m= {\mathrm e}^{-\frac{t}{2} \Delta} p(m,\cdot, \frac{t}{2}) = p(m,\cdot,t)
\]
we have that
\[
\scal{f}{K_m}_{\hh^t}= \scal{g}{p(m,\cdot,  \frac{t}{2})}_2 = f(m)
\]
so that the reproducing kernel of $\hh_t$ is precisely
$p(\cdot,\cdot,t)$. 
\end{proof}

If $M$ is compact, we have a natural characterization of the fact that
$H^s(M)$ is a reproducing kernel Hilbert space. The notation is as in Thm.~\ref{sturm-liou}. 
\begin{prop} \label{kernelexpansion}
Let $M$ be compact. Given $s\in [0,+\infty)$, the Sobolev space $H^s(M)$ is a reproducing kernel Hilbert
space if and only  if for all $m\in M$ one has
\begin{equation}
  \label{eq:41}
\sum_{k} (1+\lambda_k)^{-s} |f_k(m)|^2<+\infty.
\end{equation}
In such a case, the  reproducing kernel $K_s$ is given by
\begin{equation}
  \label{eq:42}
  K_s(m,m') = \sum_{k} (1+\lambda_k)^{-s} f_k(m) f_k(m')\qquad m,m'\in M,
\end{equation}
where the series is absolutely convergent.
\end{prop}
\begin{proof}
If $H^s(M)$ is a reproducing kernel Hilbert space, \eqref{eq:42} is
the content of the Mercer theorem, see~\eqref{eq:mercer} and, for
example, \cite[page 403]{carmeli2006vector}  taking into
account~\eqref{eq:11} and the fact that the volume measure $dm$ has support equal to $M$.
Assume now~\eqref{eq:41}. 
Define the feature map $\Phi:M\to L^2(M)$
\[ \Phi(m)= \sum_k (1+\lambda_k)^{-s/2} f_k(m) f_k,\]
  which is well defined since $\set{f_k}_k$ is a base and
  $\set{(1+\lambda_k)^{-s/2} f_k(m)}_k$ is an $\ell_2$ sequence for
  all $m\in M$.  {Denoted by $\R^M$ the vector space of functions from
  $M$ to $\R$}, we claim that the linear map $\Phi_*: L^2(M) \to \R^M$, 
\[ \Phi_*(g)(m)=\scal{g}{\Phi(m)}_2 =\sum_k (1+\lambda_k)^{-s/2} f_k(m)
  \scal{g}{f_k}_2, \] 
is injective. In fact, take $g\in L^2(M)$ such that $\Phi_*(g)=0$, {\em i.e.}
\begin{equation}
  \label{eq:12}
  \sum_k (1+\lambda_k)^{-s/2} f_k(m)
  \scal{g}{f_k}_2 = 0 \qquad \forall m\in M.
\end{equation}
Since the sequence $\set{(1+\la_k)^{-s/2}}_k$ is bounded and $\set{f_k}_k$
is a base in $L^2(M)$, then there exists $h\in L^2(M)$ such that
  \begin{equation}
h=\sum_k (1+\lambda_k)^{-s/2} 
  \scal{g}{f_k}_2 f_k .\label{eq:13}
\end{equation}
Since the series converges in $L^2(M)$,  there exists an
increasing sequence $\set{n_j}_j$ of integers such that, for almost
all $m\in M$, 
\[ \lim_{j\to\infty}\sum_{k=1}^{n_j} (1+\lambda_k)^{-s/2} f_k(m)
  \scal{g}{f_k}_2 =h(m). \]
Eq~\eqref{eq:12} implies that  $h=0$ in $L^2(M)$. By~\eqref{eq:13}
 it follows that for all indexes $k$, $ (1+\lambda_k)^{-s/2}
\scal{g}{f_k}_2=0$ and, hence,  $\scal{g}{f_k}_2=0$, so that $g=0$, as claimed. 

A standard result on reproducing kernel
Hilbert spaces, see~Thm.~\ref{feature} in Appendix or \cite[Thm. 2.4]{carmeli2006vector},   implies that the range of $\Phi$ is a reproducing
kernel Hilbert space $\mathcal K$ with reproducing kernel 
\begin{equation}
  \label{eq:14}
  K(m,m')= \scal{\Phi(m)}{\Phi(m')}_2= \sum_{k} (1+\lambda_k)^{-s} f_k(m) f_k(m')\qquad m,m'\in M.
\end{equation}
Since  $\Phi_*$ is injective,  $\Phi_*$ 
is an isometry from $L^2(M)$ onto $\mathcal K$.  Reasoning as in the
proof of injectivity, it is possibile to show that, given $g\in
L^2(M)$,  for almost all $m\in M$
\[ \big((\Id + \Delta)^{-s/2} g\big)(m)= \Phi_*(g)(m) \qquad
  \nor{\Phi_*(g)}_{\kk}= \nor{g}_2.\]
Compering with~\eqref{eq:36}, it follows  that $\Phi_*(g)\in H^s(M)$,
so that $\mathcal K= H^s(M)$ and $\nor{\Phi_*(g)}_{H^{s}(M)} =
\nor{\Phi_*(g)}_{\mathcal K}$. 
Formula~\eqref{eq:42} is  a restatement of~\eqref{eq:14}. 
\end{proof}
Since $M$ is compact the interpolation equality given
  by~\eqref{eq:16} can be also deduced by
  Proposition~\ref{prop:inter} in Appendix. For example,  given $s>0$ and $0<r<1$
\[
H^{sr} = [L^2(M), H^s(M)]_{r,2}
\]
see also \cite{fefupe16} for further results.


\section{Examples}\label{sec:examples}

{In this section, we specialize the above discussion considering in details a few examples.}
\subsection{ The Euclidean case}  

Denoting by $\cF: \cS^{\prime}(\R^n) \to \cS^{\prime}(\R^n)$  the Fourier transform on tempered distributions, standard arguments (see \cite{hor}) show
that the distributional kernel of $(1 +\Delta)^{-s} = \cF^{-1} (1+|\xi|^2)^{-s} \cF$ is given by
\begin{equation}  \label{fte}
K_s(x,y)= (2\pi)^{-n} \int_{\R^n} e^{i\xi(x-y)} (1+|\xi|^2)^{-s} d\xi.
\end{equation}

For $s>n/2$, the Sobolev space $H^s(M)$, $M=\R^n$, is an RKHS, by Theorem \ref{rkhs}. 
Its reproducing kernel is given by \eqref{fte} where the right-hand side  now is well defined as a Lebesgue integral (since $(1+|\xi|^2)^{-s}$ is in $L^1(\R^n)$). In particular, by Lebesgue dominance, it defines a continuous function of $x,y \in \R^n$. Furthermore, passing to polar coordinates $\xi=r \om$, $r>0, \om \in S^{n-1}$, integration can be explicitely performed in terms of special functions. More precisely, integration over the unit sphere gives
\begin{equation}
K_s(x,y)= (2 \pi)^{- \frac{n}{2}} |x-y|^{\frac{2-n}{2}} \int_0^\infty 
\frac{r^{n/2}}{(1+r^2)^s}  J_{\frac{n}{2} -1}(r|x-y|) dr,
\end{equation}
where $J_\nu(z)$ denotes the Bessel function of the first kind. Then integration over $r$ yields
\begin{equation}
K_s(x,y)= \frac{2^{1-s-n/2}}{\pi^{n/2} \Gamma(s)} K_{\frac{n}{2} -s}(|x-y|) 
|x|^{s-\frac{n}{2}},
\end{equation}
where $K_\nu(z)$ is the modified Bessel function of the third kind (see \cite{emot} and \cite{as}). Relevant properties of $K_\nu(z)$  are listed in \cite{as} (or see  the standard reference \cite{o}). We remark that, for $s \leq n/2$, formula \eqref{fte} remains valid if it is interpreted as an oscillating integral (see \cite{hor}), or, more classically, by Abel integration (i.e. by inserting a convergence generating factor $e^{-\epsilon |\xi|^2}$ inside the integral and letting $\epsilon \downarrow 0$ {\em after} integration).  We leave it to the interested reader to work this out in detail. Here we just recall from \cite{as} that in the limiting case $s=n/2$ there is logarithmic divergence as $x \to y$, corresponding to the well known relation
\begin{equation}
K_0(z)=( \log z^{-1}))(1+o(1)), \qquad (z \to 0).
\end{equation}
We recall that the appearance of logarithmic terms is connected with the (strong) singularity of Bessel's equation in $z=0$.
As a special case, for $n=2,s=1$, one recovers the well known formulae for the logarithmic potential theory in $\R^2$ (see e.g. \cite{folland}).

Furthermore, we recall that a standard computation (using partial Fourier transform with respect to the space variable $x \in \R^n$ and solution of an ordinary differential equation with respect to $t \in \R$) gives the heat kernel  in the explicit form
\begin{equation}
e^{t \Delta}(x,y)= (4 \pi t)^{-n/2} e^{- \frac{|x-y|^2}{4t}},
\end{equation}
the so called Gaussian kernel.


\subsection{One dimensional compact submanifolds of $\R^d$}

In this section we analyze one dimensional compact submanifolds of $\R^d$ in more detail. 
We recall that 
  for any connected compact one-dimensional sub-manifold  $M$ of $\R^d$
  of length $2\pi$  there always exists an isometry $\Psi$ from the round circle
  $(S^1,g_c)$ onto $(M,g_M)$, where $g_c$ is the Riemannian metric on
  $S^1$ induced  by the embedding of $S^1$ into $\R^2$ and $g_M$ is  the
  Riemannian metric on  $M$ induced  by the embedding of $M$ into
  $\R^d$. Here,  {\it isometry} means that $\Psi$ is a diffeomorphism from
  $S^1$ onto $M$ such that $\Psi^*g_M= g_c$. This last condition is
  equivalent to the fact that in each point $x\in S^1$ the tangent map
  $\psi_*$ is a bijective isometry from $T_xS^1$ onto
  $T_{\Psi(x)}M$.


\begin{proposition}\label{prop:equiv}
Any two one-dimensional connected compact Riemannian manifolds
$(M_j,g_j)$ are isometric (i.e. isomorphic as Riemannian manifolds) if and only if their total length (their Riemannian volume)
is equal.  In particular, any  compact one-dimensional sub-manifold $(M,g_M)$ in $\R^d$ of length 
$2\pi$ is isometric to the round sphere $(S^1, g_c)$.
If $\Psi:S^1 \to M$ denotes this isometry, then

\begin{enumerate}
\item the Riemannian measure $g_M$ is the push-forward of the
    Riemannian measure $g_c$ on $S^1$, {\it i.e.} $$\Psi_*(g_c)=g_M \;,$$ 
\item the linear map
    \[ \Psi^*=U: L^2(M,g_M)\to L^2(S^1,g_c)\qquad f\mapsto f\circ\Psi \] is a
    unitary operator, 
\item the corresponding Laplacians  are unitarily equivalent, {\it i.e.} $\Delta_M = U^* \Delta_{S^1} U$.
\end{enumerate}
\end{proposition}

\begin{proof}
First we recall that $M_1$ and $M_2$ are diffeomorphic.
This is a special case of a more general result in differential geometry: 
If an $d$-dimensional manifold $M$ carries $d$ commuting vector fields, linearly independent at each point of $M$ and having flows defined for all times 
(which is automatic if $M$ is compact), then $M$ is diffeomorphic to a product $S^k \times \R^{d-k}$ of a $k$-dimensional torus and a $d-k$-dimensional 
Euclidean plane, for some $k \in \{0, \cdots,d\}$. The diffeomorphism is basically given by the group action induced by the $d$ commuting flows, namely 
$$\R^d \ni t \mapsto g^t(m)\;,$$ 
where 
$$g^t=g_1^{t_1} \circ \cdots g^{t_d}_d:M \to M$$ is composition of the commuting flows $g^{t_j}_j$ corresponding to the 
commuting vector fields  and $m \in M$ is an arbitrary reference point (see e.g. \cite{arnold}).   Thus, for $d=1$, it suffices to pick on $M$ a 
non-vanishing
velocity field. If this vector field is chosen of unit length at each point (obtained by normalizing the field at each point), the associated diffeomorphism actually is a diffeomorphism between and $M$ and the  round sphere of length $\ell_M$ of $M$, proving our claim.

For the sake of the reader we shall prove this more explicitly by using a standard parametrization, restricting ourselves to the case of the unit sphere
$S^1$.  
Let $j:M\to\R^d$ be the embedding of $M$ into $\R^d$
and $i:S^1\to M$ be a given diffeomorphism,
given e.g. by the first argument above. Furthermore, let $x$ be the  system of coordinates on
the open set  $S^1\setminus\{x_0\} $
 \[ I\ni \theta \mapsto (\cos\theta,\sin\theta)=x(\theta), \]
where $I=(\theta_0,\theta_0+2\pi)$  for some $\theta_0\in\R$ and 
$x_0=(\cos(\theta_0),\sin(\theta_0))$.  A simple computation shows
that
\begin{equation}
\label{eq:6}
g_c(\frac{d}{d\theta} , \frac{d}{d\theta} ) =
\scal{\frac{dx(\theta)}{d\theta}}{\frac{dx(\theta)}{d\theta}}_{\R^2}
=1 \qquad \theta\in I   
\end{equation}
where $\frac{d}{d\theta}$ denotes the corresponding canonical vector
field.  Then $\theta\mapsto i(x(\theta))$ is  system of
coordinates on the open set  $M\setminus \{ i(x_0)\}$. Set 
\[
s: I\to \R\qquad  s(\theta)= \int_{\theta_0}^{\theta} \nor{
   \frac{dj(i(x(\theta')))}{d\theta'}} d\theta',
\]
where $\theta\mapsto j(i(x(\theta)))$ is  smooth on the closed interval
$[\theta_0,\theta_0+2\pi]$.  Since
\[s(\theta)'= \nor{\frac{dj(i(x(\theta)))}{d\theta}} =
\sqrt{g_M(i_*(\frac{d}{d\theta}), i_*(\frac{d}{d\theta}))}>0, \]    
where $i_*(\frac{d}{d\theta})$ is the canonical vector field associated
with the system of coordinates $i(x(\theta))$, then
$\theta\mapsto s(\theta)$ is a positive change of
coordinates from $I$ into $(0,\ell_M)$, where
\[
\ell_M= \lim_{\theta\to \theta_0+2\pi^{-}} s(\theta)
\]
is the length of $M$ since
\[ [\theta_0,\theta_0+2\pi]\ni\theta \mapsto i(x(\theta))\in M\]
is a closed simple smooth curve with range $M$.  Possibly by rescaling the
metric $g_M$, we assume that $\ell_M=2\pi$, 
so that $s(I)=(0,2\pi)$. It follows that  
\[ \varphi:(0,2\pi)\to M \qquad \varphi(t)= i(x( s^{-1}(t))) \]
is  system of coordinates on the open set  $M\setminus \{ i(x_0)\}$  and
\begin{equation}
\sqrt{g_M(\frac{d}{dt}, \frac{d}{dt})} =
\nor{\frac{dj(\varphi(t))}{dt}}= 1 \qquad \forall t\in
(0,2\pi).\label{eq:15},
\end{equation}
where $\frac{d}{dt}$ is the canonical vector field associated
with the system of coordinates $\varphi$.

Define $\Psi:S^1\to M$ such that $\Psi(x_0)=i(x_0)$ and  if $x=x(\theta)\in
  S^1\setminus{x_0}$ with $\theta\in I$, as  
  \[ \Psi(x(\theta)) =\varphi(\theta-\theta_0),\]
which is by construction a diffeomorphism.    As a consequence of
$\eqref{eq:6}$ and $\eqref{eq:15}$, we get  that $\Psi^*(g_M)=g_c$,
  which proves that $\Psi$ is an isometry. 

Note that in general $i^*(g_M) \neq g_c$. This means that to identify
$S^1$ with $M$ there is the need to choose an appropriate system of
coordinates, namely the arc-length parametrization $\varphi$ (corresponding to a unit tangent vector field).
The rest of the proof follows standard arguments and is left to the reader.

\end{proof}

We remark that a non-compact connected one-dimensional manifold $M$ still carries a unit tangent field. 
Thus, if $M$ is embedded in $\R^d$ with metric induced by the ambient space, it is necessarily of infinite length (otherwise it has endpoints and the submanifold property breaks down at the endpoints). Thus it is isometric to the real line and its Laplacian is unitarily equivalent to the standard Laplacian in $\R$. This extends Proposition \ref{prop:equiv}
to the non-compact case.

Since the length of the sphere appears in the spectrum of the Laplacian just as a scaling factor, we may confine ourselves to considering only the case of the unit sphere  
\[M=S^1 = \set{x\in \R^2 \mid \nor{x}=1} \;. \]
By Theorem \ref{rkhs}, 
the Sobolev space $H^s(M)$ is an RKHS for any $s>1/2.$ Thus, in this case, Proposition \ref{kernelexpansion}  applies and gives 
absolute convergence in a pointwise sense of the expansion \eqref{eq:42} of the reproducing kernel $K_s$ in terms of the eigenfunctions $f_k$ of the 
Laplace-Beltrami operator on $M$. The above estimate \eqref{eq:42} on the convergence of the eigenfunction expansion is far from trivial as it 
automatically implies the pointwise absolute convergence for {\em any} compact one dimensional submanifold of $\R^d$. 
Applying the (suboptimal) estimate \eqref{eq:1} of Theorem \ref{sturm-liou}, for instance, only implies a 
bound 
$$O(\lambda_k (1 + \lambda_k)^{-s})$$ 
on the individual terms of the sum in \eqref{eq:42}, and this is 
quite far from giving convergence. 

However, analyzing the kernel 
for the round sphere $S^1$, with metric induced from the Euclidean metric in $\R^2$, 
can be explicitly performed by Fourier analysis. 
In fact, the theory of the next section includes the case of the circle $S^1$ as a 
special case (provided, as remarked below, the Gegenbauer polynomial is replaced by the Chebyshev polynomial in all appropriate places).
For the sake of the reader, we shall here explicitly analyze the case of the round sphere $(S^1,g_c)$ by classical Fourier analysis.

We equip $M$  with the  system of coordinates 
 \[ I\ni \theta \mapsto (\cos\theta,\sin\theta)=x(\theta), \]
where $I\subset\R$ is any open interval of length $2\pi$. The corresponding  vector
field and one form  are denoted by $\frac{d}{d\theta}$ and
$d\theta$. Given a point $m_0=x(\theta_0)$ with $\theta_0\in I$, the map 
  \begin{equation}
 \R\ni v\mapsto v x'(\theta_0) \in \R^2\label{eq:3ex}
\end{equation}
identifies the tangent space $T_{m_0}(M)\subset \R^2$ with $\R$.
The Euclidean metric of $\R^2$ induces a Riemannian
structure on $M$ and the Riemannian tensor is
\[g_c =d\theta \otimes d\theta,\]
since  
\[
g_c(\frac{d}{d\theta} , \frac{d}{d\theta} ) =
\scal{\frac{dx(\theta)}{d\theta}}{\frac{dx(\theta)}{d\theta}}_{\R^2}
=1 \qquad \theta\in I.
  \]
We have thus explicitly checked that our coordinates actually give an isometry $x: \R/2\pi \mathbb Z \to S^1$ where by a usual abuse of notation we have identified the intervall $I$ with the manifold $\R/2 \pi \mathbb Z$.
By using the identification given by~\eqref{eq:3ex}, it is immediate to check that, given  $m_0=x(\theta_0)\in M$, the
exponential map at $m_0$ is 
\[
\exp_{m_0}:\R\to T_{m_0}(M) \qquad \exp_{m_0} (v)=x(\theta_0 + v) \qquad v\in\R,
\]
so that  the injective radius is $j(m)=\pi$ and $r_M=\pi$. 
The Riemannian volume is
\[
\int_M f(m) dm = \int_I f(x(\theta)) d\theta,
\]
where $d\theta$ is the Lebesgue measure of $I$, and   the Laplacian is
\[
\Delta f(x(\theta))= - \frac{d^2 f(x(\theta))}{d^2\theta}.
\]
For all $k\in \N$ and $i=1,2$, set $f_0,f_{k,i}:M\to \R$
\[
f_0(x(\theta)) = \frac{1}{\sqrt{2\pi}} \qquad f_{k,i}(x(\theta))=
\begin{cases}
  \frac{1}{\sqrt{\pi}} \cos(k \theta) & i=1 \\
   \frac{1}{\sqrt{\pi}}  \sin(k \theta) & i=2
\end{cases},
  \]
then $\set{f_0}\cup \set{f_{k,i}\mid k\in\N,
  i=1,2}$ is an orthonormal base of $L^2(M)$ of eigenvectors of $\Delta$
  \begin{equation}
  \Delta f_0 = 0 \qquad \Delta f_{k,i} = k^2 f_{k,i}\label{eq:4ex}
\end{equation}
and the eigenvalues of $\Delta$ are
\[
\lambda_0= 0 \qquad \lambda_{k,i}= k^2 .
  \]
Denote by $L^2(M)_0= \set{f_0}^\perp$ and $P$ the corresponding
orthogonal projection,
so that
\[L^2(M)= \R\set{f_0}\oplus L^2(M)_0\qquad   \Id=f_0\otimes f_0 \oplus P .\]
It follows that, given $s>0$, the operator  $\Delta^s$ leaves
invariant $L^2(M)_0$ and the restriction is  injective. We denote its bounded
inverse by $\Delta^{-s}$ and set
\[
A_s=f_0\otimes f_0 + P^*\Delta^{-s/2}P,
\]
where $P^*$ is the canonical isometry embedding $L^2(M)_0$ into
$L^2(M)$.  By~\eqref{eq:4ex},
  \begin{equation}
A_s = f_0\otimes f_0 +\sum_{k\geq 1,i=1,2} \frac{1}{k^{2s}}
f_{k,i}\otimes f_{k,i},\label{eq:5ex}
\end{equation}
where the convergence is in the strong operator topology.  It follows that 
$f\in H^s(M)$ if and only if there exists a (unique) $g\in L^2(M)$ such
that $ f=A_s g$. Furthermore,  $\nor{g}$ is equivalent to the Sobolev norms
$\nor{f}_{H^s,1}$, $\nor{f}_{H^s,2}$, $\nor{f}_{H^s,3}$, {\em i.e.} 
\[
H^s(M)=A_sL^2(M)\qquad  \nor{f}^2_{H^s}= \nor{g}^2=\scal{f}{f_{0}}^2+ \sum_{k\geq 1,i=1,2} k^{2s} \scal{f}{f_{k,i}}^2 <+\infty.
 \]
 By~\eqref{eq:5ex}, $A_s$ is a Hilbert Schmidt operator if and only if
 $s\geq 1/4$. Under this assumption $A_s$ is the integral operator  $L_{K_s}$
   \[
     A_s f(m) =L_{K_s}f(m)=\int_M K_s(m,m') f(m') \,dm' ,
   \]
The integral kernel is given by
   \begin{alignat}{1}
     \label{eq:2ex}
      K_s(m,m')& = f_0(m)f_0(m') + \frac{1}{\pi} \sum_{k\geq 1,i=1,2
      }\frac{1}{k^{2s}}  f_{k,i}(m)f_{k,i}(m')  \nonumber\\
      & = 1 + \frac{1}{\pi} \sum_{k\geq 1 } \frac{1}{k^{2s}}
     \cos(k(\theta-\theta')) \qquad m=x(\theta), m'=x(\theta').
   \end{alignat}
where the series converge  in $L^2(M\times M)$ and the second equality
is a consequence of 
\begin{equation}
  \label{eq:1ex}
  \sum_{i=1,2} f_{k,i}(x(\theta)) f_{k,i}(x(\theta')) = \cos(k\theta)
\cos(k\theta') + \sin(k\theta) \sin(k\theta')=\cos(k(\theta-\theta')).
\end{equation}
Furthermore if $s>1/2$, then $H^s(M)$ is a reproducing kernel Hilbert
space and the corresponding reproducing kernel is
 \[
K_s(m,m')= 1 + \frac{1}{\pi} \sum_{k\geq 1 } \frac{1}{k^{2s}}
     \cos(k(\theta-\theta')) \qquad m=x(\theta), m'=x(\theta'),
   \]
where the series  converges  normally. Note that $K_s$ is jointly continuous. 

 We now consider two cases. By a standard result on Fourier series,
 see 1.443.3 and 1.448.2 in \cite{grry07}
 \begin{subequations}
   \begin{alignat}{3}
     & \sum_{k\geq 1 } \frac{1}{k^{2}} \cos(k\theta ) \,&=&\,
     \frac{\theta^2}{4}-\frac{\pi}{2} \theta+ \frac{\pi^2}{6} \quad &&
     \quad \theta\in [0,2\pi] \label{eq:2exa}
     \\
     & \sum_{k\geq 1 } \frac{1}{k} \cos(k\theta ) \,&=& \,-\frac{1}{2}
     \ln(2(1-\cos \theta)) \quad &&
     \quad \theta\in (0,2\pi) \label{eq:2exb}
   \end{alignat}
 \end{subequations}
 where the series converge point-wisely. Hence,   given $
 m=x(\theta),m'=x(\theta' )\in M$ with $\theta'-\theta\in [0,2\pi)$,
\begin{alignat}{3}
   & K_1(m,m') \,&=&\, 1+ \frac{(\theta'-\theta)^2}{4\pi}-\frac{\theta'-\theta}{2}
   +
     \frac{\pi}{6}     \quad && \quad  \label{eq:2exc}
     \\
 & K_{1/2}(m,m') \,&=&\, 1-\frac{1}{2\pi}  \ln(2(1-\cos(\theta'-\theta)))
     \quad && \quad   \theta'\neq\theta.\label{eq:2exd}    
 \end{alignat}
 It is interesting to observe
 that  $L_{K_{1/2}}=A_{1/2}$ is a positive integral operator mapping $L^2(M)$ onto
 $H^1(M)\subset C(M)$, but its kernel $K_{1/2}$, which is defined and
 jointly  continuous on $M\times M\setminus\set{ (m,m)\mid m\in M}$, can not be
 extended to a kernel $K:M\times M\to\R$ of positive type. Indeed, assume by
 contradiction that there such a kernel. Setting $f:\R\to \R$ as 
 $f(\theta)=K(x(\theta), x(\theta))$,   since $K$ is of positive
 type, then for all $\theta \neq\theta'$
 \[
f(\theta) f(\theta')\geq \left( 1-\frac{1}{2\pi}  \ln(2(1-\cos(\theta'-\theta)))\right)^2.
 \]
Fix $\theta'=\theta_0\in \R,$ and  take the limit for $\theta$ going to
$\theta_0$, then 
 \[
\lim_{\theta\to \theta_0} f(\theta) =+\infty \qquad\forall\theta_0\in \R,
   \]
   which is impossible. Indeed, set $I_n= f^{-1}( (-\infty,n])$ with
   $n\in\mathbb Z$. If
   $I_n$ has a cluster point $\theta_0\in\R$, then
   $\lim\inf_{\theta\to\theta_0} f(\theta)\leq n$, which is
   impossible. Then $I_n$ is countable, but $\R= \cup_{n\in \mathbb
   Z} I_n$, which is impossible since $\R$ is not countable. 

 Note that $L_{K_{1/2}}= L_{K_1}^{1/2}$ and $L_{K_1}$ is an integral operator with
 a Mercer kernel. This provides an alternative counter-example to the
 construction provided in  \cite{zh02} about the existence of a 
 Mercer kernel $K$ such that  $L_{K}^{\frac{1}{2}}$ is an integral
 operator whose kernel is not of
 positive type. Furthermore, $L_{K_{1/2}}$ is a positive operator with range
 into $C(M)$, but its kernel is not of positive type.  Observe that,
 if the kernel of a positive integral operator is jointly continuous,
 then $K$ is of positive type by Theorem~2.3 in \cite{feme09}.

\subsection{The unit sphere $S^{d-1}$}
\paragraph{Basic Facts.}

For any $d\geq 3$ we denote by $M=\set{ x\in\R^d \mid \nor{x}=1}$ the unit sphere in $\R^d $, equipped with the Riemannian metric induced from the euclidean metric in the ambient space and the associated 
surface measure $d\sigma $. Then the Laplace-Beltrami operator $\Delta_M = \Delta$ is a classical differential operator which arises e.g.  by transforming unitarily the 
Laplace operator $-\Delta_{\R^d}=-\sum{ \partial_i^2} $ to polar coordinates: If
$$ {\cal U}: L^2(\R^d, dx) \to L^2(\R_+ \times M, dr d\sigma), \quad
{\cal U}f(r,m):= r^{\frac{d-1}{2}} f(r m) $$

denotes the unitary transformation to polar coordinates, one finds

\begin{equation}  \label{Delta}
{\cal U} (-\Delta_{\R^d}) {\cal U}^{-1} = -\partial^2_r + \frac{\Delta + \alpha}{r^2}, \qquad \alpha= \frac{1}{4}(d-1)(d-3).
\end{equation}
The eigenspace $H_\ell$ of $\Delta$  in $L^2(M)$ for the eigenvalue 
$\lam_\ell=\ell(\ell+d-2)$ is precisely given by the restrictions to the unit sphere of harmonic polynomials in $\R^d$, homogeneous of degree $\ell$, which commonly are 
called {\it spherical harmonics of degree $\ell$}.
The degree $\ell $ is a natural parameter for the eigenspace $H_\ell$ and the eigenvalue $\lam_\ell$. Via the relation \eqref{Delta} all the eigenvalues
$\lam_\ell$ and the dimensions $d_\ell= \dim H_\ell$ can be explicitly calculated using the euclidean Laplacian $-\Delta_{\R^d}$ in $\R^d$, completely avoiding the use of local 
coordinates for $M$. One finds

\begin{equation}  \label{dim}
d_\ell = (2 \ell + d- 2) \frac{\Gamma(\ell + d-2)}
{\Gamma(d-1) \Gamma(\ell +1)},
\end{equation}

which we consider as a meromorphic function of the complex parameter $\ell \in \C$ in view of the basic properties of the Gamma-function $\Gamma(z)$.
Also note for further use that the eigenvalues $\lam_\ell$ are analytic functions of $\ell$. These analyticity properties will be important for our analysis. 
For completeness sake we recall that in view of these formulae $\Delta + \alpha$
on $H_\ell$ is multiplication by

\begin{equation} \label{cent}
\la^2 - \frac{1}{4}, \qquad \la= \ell + \nu, \quad \nu=\frac{d}{2} -1,
\end{equation}

which gives the 
form of the  centrifugal barrier in formula \eqref{Delta}. 
For these and most of the subsequent formulae we refer to any good book on PDE 
like e.g. \cite{folland} for basics, the old group theoretic treatment in \cite{bmp1,bmp2} and in particular the review in \cite{ak} which we shall largely follow in spirit. 
We also mention the work \cite{kalf95} which treats the expansion of a differentiable function on the sphere in terms of spherical harmonics. The main difference to our  approach is that it is essentially {\em real} (both
the parameters $\ell$ and $\inner{m, m'} = -\cos \theta$ take exclusively real values) while our approach is essentially complex, using analyticity.
Roughly speaking, the real approach is fine to treat absolutely and uniformly converging series, while a complex approach seems much better adapted to handle divergent series by Abel summation and to treat kernels with singularities.
\\\\
We choose a real orthonormal basis of spherical harmonics $f_{\ell,k}$, with $1 \leq k \leq d_\ell,$ of $H_\ell$. Then the addition formula for spherical harmonics  expresses 
the orthogonal projection $\Pi_\ell$ on $H_\ell$ in the Hilbert space ${\cal H}=L^2(M,d\sigma)$, for $d \geq 3$, in terms of  the kernel
\begin{equation}  \label{addition}
\Pi_\ell(m,m') = \sum_{k=1}^{d_\ell} f_{\ell,k}(m) f_{\ell,k}(m')=
\frac{d_\ell}{\om_d} B^\nu_\ell(\inner{m ,m'}), \qquad 
B^\nu_\ell(z)= \frac{C^\nu_\ell(z)}{C^\nu_\ell(1)}, \quad 
\nu=\frac{d}{2} -1,
\end{equation}
$$ (\Pi_\ell f)(m)= \int_{S^{d-1}} \Pi_\ell(m,m') f(m') d\sigma(m'),$$
where
$$ \om_d= \mbox{vol} (S^{d-1}) = 2 \pi^{d/2} \Gamma \paren{\frac{d}{2}}^{-1}  $$
is the volume of the sphere $M$, and the Gegenbauer (or ultra-spherical) polynomial $C^\nu_\ell(z)$ is defined for $\nu >0$ by use of its generating function through the identity

\begin{equation}
(1-2zt + t^2)^{-\nu} =: \sum_{\ell=0}^\infty C^\nu_\ell(z) t^\ell.
\end{equation}

We refer to $B^\nu_\ell(z)$ as the normalized Gegenbauer polynomial. It is expressed in terms of the hypergeometric function $F(a,b,c;t)$ as

\begin{equation} \label{gegenb}
B^\nu_\ell(z)=F(\ell+2\nu, -\ell,\nu+1/2; \frac{1-z}{2}).
\end{equation}

The rhs of \eqref{gegenb} extends as an analytic function to any $\ell \in \C$ and $\nu \geq 0$, and henceforth we shall denote by $B^\nu_\ell(z)$ this extension provided 
by the hypergeometric function. In particular one obtains

\begin{equation}
B^0_\ell(z)=T_\ell(z), \qquad (\ell \in \N),
\end{equation}

where $T_\ell(z)$ is the Chebyshev polynomial $T_\ell(\cos \theta)=\cos(\ell \theta).$ With this definition of $B^0_\ell(z)$ the addition 
theorem \eqref{addition} also holds in dimension $d=2$.

The normalized Gegenbauer polynomial verifies

\begin{equation}
B^\nu_\ell(z)= (-1)^\ell  B^\nu_\ell(-z), \qquad |B^\nu_\ell(z)| \leq 1 \;,\quad -1 \leq z \leq 1.
\end{equation}

Now, using an expression of the associated Legendre function $P^\mu_\la(z)$ for $-1<z<1$ in terms of the hypergeometric function and one of the transformation 
identities for $F(a,b,c;t)$ one obtains, for $\ell \in \N$ and $\nu \geq 0$,

\begin{equation}
B^\nu_\ell(z)= 2^{- \mu} \Gamma(1-\mu) (1-z^2)^{\mu/2} P^\mu_\la(z)\;, \qquad
\mu= \frac{1}{2} - \nu\;,  \quad \la=\ell + \nu -\frac{1}{2} \;. 
\end{equation}

Then an integral representation of $P^\mu_\la(z)$ can be used to finally obtain the following integral representation of the normalized Gegenbauer polynomial

\begin{equation}    \label{intrep}
B^\nu_\ell(\cos \theta) = 2^\nu \frac{\Gamma(\nu +\frac{1}{2})}{\Gamma(\nu)
\Gamma(\frac{1}{2})} (\sin \theta)^{1-2\nu} \int_0^\theta 
\cos(\ell + \nu) \phi \left ( \cos \phi - \cos \theta \right )^{\nu -1} 
d\phi,
\end{equation}

valid for $ \ell \in \C, \nu \geq 0$ and $ 0 <\theta < \pi$, see \cite{ak}. Although the derivation of \eqref{intrep} is classical (based on \cite{bmp1, bmp2}) it is 
possibly not a well known formula, at least compared to the basic identities for the Gegenbauer polynomial which appear in many textbooks.

\paragraph{Analyzing the kernel of $(1+\Delta^s)^{-1}$.}
With these preparations, it is possible to analyze the kernel $K_s(m,m')$ of $(1+\Delta^s)^{-1}$ by explicit computation. At least formally, one has

\begin{equation}   
\label{eq:kernelKs}
K_s(m,m')= \sum_{\ell=0}^\infty (1+\lam^s_\ell)^{-1} \Pi_\ell(m,m').
\end{equation}

Thus, using the asymptotic relation $d_\ell=O(\ell^{d-2})$ and 
$\lam_\ell=O(\ell^2)$ as $\ell \to \infty$, one obtains for the summands on the rhs of \eqref{eq:kernelKs} the estimate
$$ O(l^{-2s+d-2})=O(\ell^{-1-2\epsilon}),   \qquad \mbox{for  }  s=\frac{d-1}{2} + \epsilon, $$
which proves convergence of the expansion \eqref{eq:kernelKs} for $s> \frac{d-1}{2}$ as predicted by our 
Proposition \ref{kernelexpansion}, which  for $s> \frac{d-1}{2}$  identifies the formal expansion with the kernel of $(1+\Delta^s)^{-1}$.

Noting that $B^\nu_\ell(1)=1$, we also obtain from \eqref{eq:kernelKs} that the kernel $K_s(m,m')$ 
diverges on the diagonal for $s \leq \frac{d-1}{2}$, tending to $\infty$. 
\\
\\

\noindent
We shall now show that in the complementary case $0 < s \leq \frac{d-1}{2}$ Abel summation of 
the then divergent sum in \eqref{eq:kernelKs}, combined with the integral representation \eqref{intrep}, can be used to show 
that the distributional kernel 
$K_s(m,m')$ of $(1+\Delta^s)^{-1}$ is smooth (in fact real analytic) away from the diagonal and to bound the singularity on the diagonal $m=m'$.
Technically, the crucial point is to realize the individual terms in the sum on the rhs of \eqref{eq:kernelKs} as residues of an 
appropriate meromorphic function, allowing to rewrite the sum as a contour integral in the complex $\ell-$ plane. 
This so called Sommerfeld-Watson transformation has been popular in the physics literature for analyzing the partial wave expansion in 
dimension $d=3$, see e.g. \cite{ne} and \cite{so}. 
 Using the formulae of this section, the method also works in general, i.e. for all $d \geq 2$. 
We have

\begin{thm}
For $0 < s \leq \frac{d-1}{2}$, and $m \neq m'$,  the distributional kernel $K_s(m,m')$ of $(1+\Delta^s)^{-1}$ is given by Abel summation of \eqref{eq:kernelKs}, i.e.

\begin{equation}  \label{abel}
K_s(m,m') = \lim_{t \uparrow 1} K_{s,t}(m,m'),  \quad
K_{s,t}(m,m')= \sum_{\ell=0}^\infty (1+\lam^s_\ell)^{-1} t^\ell \Pi_\ell(m,m').
\end{equation}
It is real analytic in this region and satisfies the bound
\begin{equation} \label{bound}
|K_s(m,m')| =O(|m - m'|^{2s - d+1})\;, \quad 2s <d-1 
\end{equation}
as $m \to m'$, while for $2s =d-1$ the rhs is replaced by $O(|\log|m -m'||)$.
\end{thm}


\begin{proof}
We give a complete proof only for $d=3$ and indicate the additional work needed for the general case.
Writing 
\[ a_\ell(s;m,m')=(1+\lam_\ell^s)^{-1}\Pi_\ell (m,m') \;,\]
and using Lemma 6.1 from \cite{ak}, we find that this is bounded by  
$O( \ell^{-2s + d-2})$ for $d=2$ and $d \geq 3, m= \pm m',$  and (since $d-2=2 \nu$) by $O(\ell^{-2s + \nu}), \nu=\frac{d}{2} -1,$  for $d \geq 3, m \neq \pm m'$, respectively. 

In particular, the power series  in \eqref{abel} defining $K_{s,t}(m,m')$ converges for $t<1.$ We shall now show that $t=1$ is a regular point of this power series  
by rewriting it as a contour integral, using the residue theorem. From this we shall prove Abel summability.
Observe 
$$ \mbox{Res} |_{t=\ell} \left( \frac{1}{\sin \pi t} \right) = \frac{1}{\pi}  (-1)^\ell $$
and the estimate
\begin{equation}  
\label{eq:lemmagb}
|B_\ell^\nu(\cos \theta)| \leq C_{\nu,\epsilon} e^{| \Ima \ell| \theta}\;, \qquad 
\ell \in \C\;, 0< \nu \;, 0 \leq \theta \leq \pi - \epsilon \;, 0< \epsilon<\frac{\pi}{2}\;, 
\end{equation}
see Lemma 6.1 in \cite{ak}. We claim
\begin{equation}  
\label{eq:intrepabel}
K_{s,t}(m,m')= \frac{1}{2i \omega_d} \int_\Gamma (1 + \lam_\ell^s)^{-1} \frac{t^\ell d_\ell}{\sin \pi \ell} B^\nu_\ell(-m \cdot m') d\ell \;, \qquad m \neq m' \;.
\end{equation}
Here 
$\Gamma$ is the complex contour consisting of the imaginary axis for $|\ell | \geq 1/2$ and the half-circle $|\ell| =1/2$, $\Real \ell < 0$, traversed from
$-i \infty$ to $i \infty$. 
To prove \eqref{eq:intrepabel}, we denote by $\gamma_M$ the complex contour consisting of the half circle in the right half-plane $\Real \ell > 0$ of radius
$M+1/2$ for $M \in \N$. Then, using \eqref{eq:lemmagb} and the above bound for $d_\ell$, we find, 
setting $- m \cdot m' = \cos \theta$, with $ 0 \leq \theta  \leq \pi - \epsilon$ for some $\epsilon > 0$,
\begin{align*}
\int_{\gamma_M} (1 + \lam_\ell^s)^{-1} \frac{t^\ell d_\ell}{\sin \pi \ell} B^\nu_\ell(-m \cdot m') d\ell &=
\int_{\gamma_M} t^{\Real \ell} O(M^{d-2}) O(e^{-|\Ima \ell|(\pi - \theta)}) d|\ell| \\
&=
O(M^{d-1} t^M + M e^{- \epsilon M}) \;,
\end{align*}
which is $o(1)$ as $M \to \infty$, for any $t<1$. Thus \eqref{eq:intrepabel} follows by applying the residue theorem to the region bounded by $\Gamma$ and $\gamma_M$ and letting $M$ tend to infinity.
By \eqref{eq:lemmagb}, the integral in \eqref{eq:intrepabel} converges up to $t=1$. Thus the formal series 
in \eqref{eq:kernelKs}  is indeed Abel summable for $m \neq m'$.

 Furthermore, standard arguments give the first equality in \eqref{abel}, i.e. the identification of the limit   $ \lim_{t \uparrow 1} K_{s,t}(m,m')$ with the distributional 
kernel $T(m,m'):=K_s(m, m')$ of $T:= (1+\Delta^s)^{-1}$.
In fact, the bounded operators $T= \sum_0^\infty a_\ell \Pi_\ell, 
T_t:= \sum_0^\infty t^\ell a_\ell \Pi_\ell$ are both limits in operator norm
of their partial sums $S_n, S_{t,n}$ with smooth kernels  $S_n(m,m')$ and
$S_{t,n}(m,m')$, respectively. Since $\lim_{t\uparrow 1}T_t=T$ in operator norm, we in particular have $\lim_{t\uparrow 1}\langle u,T_t v \rangle= \langle u,T v \rangle $ for $u,v \in {\cal D}={\cal C}^\infty(S^{d-1})$. 
 Using the habitual abuse of notation, the (smooth) kernels $S_n(m,m')$ and
$S_{t,n}(m,m')$ converge in the sense of distributions (i.e. in the usual weak topology of continuous linear functionals 
${\cal D} \times {\cal D} \to \C,$ where ${\cal D}$ is equipped with its natural Frechet topology) to the distributional kernels $T(m,m')$ and $T_t(m,m')$, respectively. In our only very mildly singular case the former is a distribution of order zero, while the latter is smooth (for $t<0$). 
Also
\begin{equation}  
\label{eq:limdistr}
\lim_{t \uparrow 1} T_t(m,m')=T(m,m')
\end{equation}
in the sense of distributions (since $T_t \to T$ in operator norm). In addition, since for $v \in {\cal D}$ the $L^2-$ norm $||\Pi_\ell v||$ decays faster than any polynomial, we 
obtain $\Lambda^j T v\ \in L^2(S^{d-1})$ for all $j \in \N$. Thus, by Sobolev embedding, $T$ may be viewed as a continuous map
${\cal D} \to {\cal D}$. Consequently, we may fix $m \in S^{d-1}$ and still obtain \eqref{eq:limdistr}, now in the sense of distributions on $S^{d-1}$.
Thus, even for fixed $m \in S^{d-1}$, we may represent (as a distribution in ${\cal D}^\prime  (S^{d-1}$)) the kernel  as

\begin{equation}
T(m,m')= \lim_{t \uparrow 1} T_t(m,m')= \lim_{t \uparrow 1} K_{s,t}(m,m'),
\end{equation}
where $K_{s,t}(m,m')$ is given by \eqref{eq:intrepabel} with the integral on the rhs taken in weak sense, i.e. applied to $v \in {\cal D}$ 
(this requires redoing the residue argument for $T_t v$). For $m \neq m'$, this integral has a pointwise sense by the above result on 
Abel summability, and this finally identifies, for $m \neq m'$, $\lim_{t \uparrow 1} K_{s,t}(m,m')$ with the distributional kernel 
of $T$. We also have obtained the representation
\begin{equation}  \label{intrepKs}
K_{s}(m,m')= \frac{1}{2i \omega_d} \int_\Gamma (1 + \lam_\ell^s)^{-1} \frac{ d_\ell}{\sin \pi \ell} B^\nu_\ell(\cos \theta) d\ell \qquad(-m \cdot m'=\cos \theta \in [0,\pi)),
\end{equation}
where $B^\nu_\ell(\cos \theta)$ is given by \eqref{intrep}.
Using Fubini and interchanging the order of integration  gives the estimate

\begin{equation} 
\label{eq:est}
|K_{s}(m,m')| \leq C | \sin \theta |^{1-2\nu} \int_0^\theta I(\phi)
 \phi \left ( \cos \phi - \cos \theta \right )^{\nu -1} 
d\phi,
\end{equation}
where

\begin{equation} 
\label{eq:I}
I(\phi)=   \int_\Gamma |1 + \lam_\ell^s|^{-1} |d_\ell | \; \left|\frac{ \cos(\ell + \nu)\phi }{\sin \pi \ell}\right| \; |d\ell|   
 , \qquad C=\left|\frac{2^{\nu-1}}{ \omega_d}  \frac{\Gamma(\nu +\frac{1}{2})}{\Gamma(\nu)
\Gamma(\frac{1}{2})} \right|.
\end{equation}

We remark that this estimate, with absolute value taken inside the integral, is  sufficiently sharp only in dimension $d=3$ (where $\nu = 1/2$ 
and the prefactor in \eqref{eq:est} may be omitted). In the general case one may still apply Fubini, but one needs to 
carefully take into account oscillations in the integrand to improve the estimate. We leave this to the interested reader.

Now observe that for $\gamma +1>0$ and $\beta
>0$ one has

\begin{equation}  
\label{eq:elementary}
\int_0^\infty x^\gamma e^{-\beta x} dx=\beta^{-\gamma-1} \int_0^\infty t^\gamma e^{-t} dt =O(\beta^{-\gamma-1}) \qquad (\beta \downarrow 0).
\end{equation}

Thus, using the usual bounds in the integrand in \eqref{eq:I} and setting
$\gamma=d-2-2s$ (corresponding to $-\gamma-1=2s+1-d=:\alpha$) we get
$1-d<\alpha \leq 0$ for our parameter range $0<s\leq \frac{d-1}{2}$, and combining \eqref{eq:est}, \eqref{eq:I}, \eqref{eq:elementary} gives

\begin{equation}  
\label{estim}
|K_{s}(m,m')| \leq C |(\sin \theta)|^{1-2\nu} J(\theta), \qquad
J(\theta)=\int_0^\theta (\pi-\phi)^\alpha
  \left ( \cos \phi - \cos \theta \right )^{\nu -1} 
d\phi,
\end{equation}

for some constant $C<\infty$.
Clearly it suffices to estimate $J(\theta)$ for $\theta \uparrow \pi$, and we 
may, up to an irrelevant additive constant, replace the lower bound $0$ in the integral defining $J(\theta)$ by $\pi/2$.
Using the elementary trigonometric identity
\[
 \cos \phi - \cos \theta = 2\sin \paren{\frac{\phi - \theta}{2}} \sin \paren{\frac{\phi + \theta}{2}}\;, \qquad \frac{\pi}{2} \leq \phi \leq \theta \leq \pi  \]
and setting $s:=\pi - \phi \geq t:=\pi - \theta$, we obtain

\begin{equation}  
\label{eq:It}
J(\theta) \leq C \int _t^{\pi/2} s^\alpha (s^2-t^2)^{\nu-1} ds. 
\end{equation}

Now we first estimate in the non-critical case $2s < d-1=2$, emphasizing once again that the estimate in \eqref{eq:est} is only 
sharp in dimension 3. Then $\nu -1=-1/2$ and \eqref{eq:It} gives

\begin{equation}  
\label{eq:iIt}
I(t)= \int _t^{\pi/2} s^{\alpha -1} \paren{1-\paren{t/s}^2}^{\nu-1} \;ds  \;.
\end{equation}
Now substitute $x=t/s$ and observe $ds/s= - dx/x$ to get

\begin{equation} 
\label{eq:Itfinal}                          
I(t)= t^\alpha \int_{\frac{2t}{\pi}}^1 x^{-\alpha-1} (1-x^2)^{-\frac{1}{2}} dx = O(1) t^\alpha,
\end{equation}
where we split the range of integration into $J_1=[\frac{2t}{\pi},\frac{1}{2}]$ and $[\frac{1}{2},1]$, and then use $(1-x^2)^{-\frac{1}{2}}=O(1)$ and $-\alpha >0$ on $J_1$ combined with $(1-x^2)^{-\frac{1}{2}}$ being integrabel on $J_2$.
Clearly, $t=\pi - \theta$ is equivalent to $|m -m'|$ as $t \to 0$.
Thus, combining \eqref{eq:Itfinal}, \eqref{eq:It} and \eqref{eq:est} proves the estimate \eqref{bound}in case $\alpha <0$. 
In the critical case $\alpha=0$, we proceed similarly, but now observe that the integral over $J_1$ diverges logarithmically 
as $t \downarrow 0$. The rest of the assertions of the Theorem, in particular the statement on real analyticity, follow similarly to the arguments in \cite{ak}.
\end{proof}

Note that the singularity in equation \eqref{bound} coincides precisely with the  singularity of the Newtonian potential (or the resolvent kernel for the Laplacian) 
in $\R^n$ (with $s=1$ and $n=d-1$). It coincides with the bounds which we obtained for the kernel in Example 1 (i.e. for $\Delta$ in $\R^d$) in terms of the 
modified Bessel function.
This indicates that the singularity of the kernel is basically a local property. 



\section*{Acknowledgments}
NM is supported by the German Research Foundation under  DFG Grant STE 1074/4-1.
 L. R. acknowledges the financial support of the AFOSR projects FA9550-17-1-0390 and BAA-AFRL-AFOSR-2016-0007 (European Office of Aerospace Research and Development), and the EU H2020-MSCA-RISE project NoMADS - DLV-777826.

\bibliographystyle{abbrv}
\bibliography{biblio,biblio_ex}

\appendix

\section{Reproducing Kernel Hilbert Spaces}\label{app:rkhs}

In this section we provide the basics about reproducing kernel Hilbert spaces (RKHSs). Classical references on the topic include \cite{as}. Here,  we mainly  follow \cite{as}, \cite{Ber_book} , \cite{stechr2008}.

\subsection{Basic definitions and results}


Let $\cX \not = \emptyset $. We recall that a map $K: \cX \times \cX \longrightarrow \mbr$ is called {\it positive semi-definite} if for any $n \in \mbn$, $\alpha_1, ..., \alpha_n \in \mbr$ 
and for any $x_1,...,x_n \in \cX$ one has 
\[  \sum_{i,j=1}^n \alpha_i \alpha_j K(x_i, x_j) \geq 0 \;. \] 
If equality holds only for $\alpha_1=...=\alpha_n =0$ for distinct $x_1,...,x_n$, then $K$ is said to be {\it positive definite}.  
The map $K$ is symmetric if $K(x,x')=K(x',x)$ for any $x,x' \in \cX$. 
\\
\\
It is well known that to every symmetric positive semi-definite function $K$ one can associate a Hilbert space 
$(\cH, \inner{\cdot, \cdot}_\cH)$, called {\it feature space} and a map 
$\Phi: \cX \longrightarrow \cH$, called {\it feature map} such that 
\[ K(x,x') = \inner{\Phi(x), \Phi(x')}_\cH  \]
for any $x, x' \in \cX$. A map satisfying the latter condition is called a {\it kernel}.

\paragraph{The RKHS associated to a kernel.} 
If $\cH$ is a Hilbert space of functions $f: \cX \longrightarrow \mbr$, then $K: \cX \times \cX \longrightarrow \mbr$ is said to be a {\it reproducing kernel of $\cH$} if 
for any $x \in \cX$ we have $K(x, \cdot) \in \cH$ and if the {\it reproducing property}
\[  f(x) = \inner{f, K(x, \cdot)}_\cH   \]
holds for any $f \in \cH$ and for any $x \in \cX$. 
Note that any reproducing kernel is also a kernel in the above given sense. More precisely, we have 

\begin{proposition}
If $\cH$ is a Hilbert function space over $\cX$ with reproducing kernel $K$, then $\cH$ is an RKHS, being also a feature space of $K$ with {\it canonical feature map}  
$\Phi(x) = K(x, \cdot)$,  $x \in \cX$. 
\end{proposition}

\begin{definition}[RKHS]
The space $\cH$ is called a {\it reproducing kernel Hilbert space} over $\cX$ if for any $x \in \cX$ the evaluation functional 
$\delta_x: \cH \longrightarrow \mbr$ is continuous, i.e. 
\[ |\delta_x(f)| = |f(x)| \leq C_x ||f||_{\cH}  \]
for any $f \in \cH$ and for some $C_x >0$. 
\end{definition}
As a consequence of this definition, one finds that if two functions $f,g$ are identical as elements in $\cH$, they coincide at any point:
\[ |f(x) - g(x)| = |\delta_x(f-g)| \leq C_x||f-g||_\cH \;. \]
We have the following fundamental result:

\begin{theorem}
Every RKHS $\cH$ over $\cX$ admits a unique reproducing kernel $K$ on $\cX$, given by 
\[ K(x,x') = \inner{\delta_x, \delta_{x'}}_\cH  \;,\qquad x,x' \in \cX  \;,\]
identifying via Riesz $\delta_x$ with an element in $\cH$. 
Additionally, if $(f_k)_{k \in I}$ is an orthonormal basis of $\cH$, then 
\[  K(x,x' ) = \sum_{k \in I} f_k(x) f_k(x') \;,\qquad x,x' \in \cX   \;. \]
\end{theorem}

Conversely, any kernel has a unique RKHS:

\begin{theorem}\label{feature} 
If $K$ is a kernel over $\cX$ with feature space $\cH$ and feature map $\Phi: \cX \longrightarrow \cH$, then the space 
\[ \tilde \cH = \left\{ f:\cX \longrightarrow \mbr| \exists h \in \cH \mbox{ s.th. } f=\inner{h, \Phi(\cdot)}_\cH  \right\} \]
equipped with the norm
\[ ||f||_{\tilde \cH} = \inf\{ ||h||_\cH \;| \; h \in \cH  \mbox{ s.th. } f=\inner{h, \Phi(\cdot)}_\cH  \} \]
is the only RKHS for which $K$ is a reproducing kernel. 
\end{theorem}

Thus, there is a one-to-one relation between kernels and RKHSs.

\subsection{Mercer's Theorem and Extensions}

Assume that $(\cX, d)$ is a compact metric space possessing a finite
Borel measure $\nu$ such that its support $\supp{\nu}=\cX$. Let $\cH$ be an RKHS on $\cX$ 
with continuous kernel $K: \cX \times \cX \longrightarrow \mbr$, being bounded by compactness. The integral operator $L_K: L^2(\cX, \nu) \longrightarrow L^2(\cX, \nu)$ 
defined by 
\[  L_Kf = \int_\cX K(\cdot, x') f(x') d\nu(x')  \]
is bounded, nuclear, selfadjoint (by symmetry of $K$) and even positive. 
In particular, $L_K$ maps continuously into $\cC(\cX)$, the space of continuous functions on $\cX$.  The spectral theorem ensures the existence of an at most 
countable family $(f_k)_{k \in I}$ of functions, forming an orthonormal system (ONS) in $L^2(\cX, \nu)$ such that  for any $f \in L^2(\cX, \nu)$
\[ L_Kf = \sum_{k \in I}\lam_k \inner{f, f_k}_{L^2}f_k \;.   \]
The family $(\lam_k)_{k \in I}$ are the nonzero eigenvalues of $L_K$, counted with geometric multiplicities. Note that we may choose continuous functions as representatives of the 
eigenvectors, i.e. $f_k \in \cC(\cX)$. The classical version of {\it Mercer's Theorem} shows that the kernel $K$ enjoys a representation in terms 
of the eigenvalues and eigenfunctions, i.e., for any $x,x' \in \cX$ one has the expansion  
\begin{equation}
\label{eq:mercer}
  K(x,x') = \sum_{k \in I} \lam_k f_k(x)f_k(x') \;,
\end{equation}  
where the convergence is absolute and uniform. Such a representation as in \eqref{eq:mercer} is called a {\it Mercer representation of $K$}. 
\\
\\
The classical Mercer Theorem has been extended, relaxing the compactness of $\cX$:  
Let $J: \cH \longrightarrow L^2(\cX, \nu)$ denote the inclusion. In general, this map is not injective and thus the family 
$(\sqrt \lam_k f_k)_{k \in I}$ is not an orthonormal basis (ONB) of $\cH$ and $K$ does not have a pointwise convergent expansion \eqref{eq:mercer}. 
The next Proposition characterizes pointwise  convergent Mercer representations.

\begin{prop}[\cite{SteinScov12}, Thm. 3.1]
Let $\cX$ be a measurable space equipped with a measure $\nu$. Assume the RKHS $\cH$ possess a measurable kernel $K$ on $\cX$ and is 
compactly embedded into $ L^2(\cX, \nu)$. Then $K$ admits a pointwise convergent Mercer representation  \eqref{eq:mercer} if and only if 
the operator $J: \cH \longrightarrow L^2(\cX, \nu)$ is injective.
\end{prop}

\begin{prop}[\cite{SteinScov12}, Cor. 3.5]
Let $\cX$ be a Hausdorff space and $\nu$ be a Borel measure on $\cX$. Moreover, let $K$ be a continuous kernel whose RKHS is compactly embedded into  $ L^2(\cX, \nu)$. 
Then the convergence of 
\[    K(x,x') = \sum_{k \in I} \lam_k f_k(x)f_k(x') \]
is uniform in $x$ and $x'$ on every compact subset $A \subset supp(\nu)$. 
\end{prop}


\subsection{Relation to Interpolation spaces}
\label{sec:interpolation}

{ The fractional powers of the integral
operator $L_K$ are defined by~\eqref{eq:4} with the choice
$\Phi(\la)=\la^r$. Since $L_K$ is compact, we have a more explicit formula.}
Let $(f_k)_{k \in I}$ be an ONS in $ L^2(\cX, \nu)$,
consisting of eigenfunctions  
of $L_K$ associated to $(\lam_k)_{k \in I}$. 
Given $r \in [0, \infty)$, the power $L_K^r:  L^2(\cX, \nu) \longrightarrow L^2(\cX, \nu)$  
is given by   
\[  L^r_K f := \sum_{k\in I} \lam_k^r \inner{f, f_k}_{L^2} f_k \;, \qquad f \in L^2(\cX, \nu) \;. \]
Note that this definition is independent of the chosen ONS of eigenfunctions. Then $L_K^r$  can be identified with an integral operator corresponding to a new kernel. 
We summarize some results given in \cite{SteinScov12}.  

\begin{prop}
Let $\cX$ be a measurable space with measure $\nu$ and $K$ be a measurable kernel on $\cX$ whose RKHS $\cH$ is compactly embedded into $L^2(\cX, \nu)$. Then 
$L^r_K = L_{K^r}$, where for any $x,x' \in \cX$ one has 
\[  K^r(x, x') = \sum_{k \in I} \lam^r_k f_k(x) f_k(x') \;. \] 
The power $K^r$ is a kernel with associated RKHS $\cH^r$, provided 
\[  \sum_{k \in I} \lam^r_k f_k^2(x) < \infty \;, \qquad  x \in \cX \;.  \] 
Moreover, $\cH^r$ is separable and compactly embedded into $L^2(\cX, \nu)$, 
satisfying $\cH^{r_1} \hookrightarrow  \cH^{r_2}$, $r_2 \leq r_1$. 
\end{prop}

Let $E, F$ be two Banach spaces which are continuously embedded in some topological (Hausdorff) vector space $\cE$. For $0<r<1$ and $1\leq \beta \leq \infty$ 
we denote by $[E,F]_{r,\beta}$ the interpolation space, defined by the {\it real interpolation method}, see e.g. \cite{Ben88}. 
The images of the above defined power spaces can be identified with interpolation spaces. 

\begin{prop}[\cite{SteinScov12}, Thm. 4.6]
\label{prop:inter}
For any $0<r<1$ one has $ran(L_K^{r/2}) = \cH^{r} =[L^2(\cX, \nu), \cH]_{r, 2}$. 
\end{prop}

\section{Basic notions on Riemannian manifolds}\label{review}

In this section we review the definitions and  results on Riemannian
manifolds, which are needed in the paper, see \cite{pet16} as a standard
reference. 
\begin{prop}\label{prop_M}
Let $M$ be an $n$-dimensional connected Riemannian manifold. 
  \begin{enumerate}[a)]
  \item\label{distance} The manifold $M$ has a natural structure of metric space with
    respect to the distance
    \[ d(m,m')= \inf_{\gamma} \ell_\gamma,\] where the infimum is
    taken over all the smooth curves $\gamma:[a,b]\to\R$ such that
    $\gamma(a)=m$ and $\gamma(b)=m'$, and $\ell_\gamma$ is the length of
    $\gamma$, {\em i.e.} 
    \[ \ell_\gamma =\int_a^b
      \sqrt{g(\gamma'(t),\gamma'(t))}_{\gamma(t)}\, dt.\]
  \item\label{Hopf} The manifold $M$ is complete if one of the
    following equivalent conditions is satisfied:
    \begin{enumerate}[i)]
    \item the space $M$ is complete as a metric space;
    \item the closed and bounded subsets of $M$ are compact;
    \item for all $m\in M$ and $v\in T_m(M)$ there exists a unique smooth
      curve $\gamma:[0,+\infty) \to M$, called geodesic, such that
      \begin{equation}
        \gamma(0)= m \qquad \gamma'(0)=v\qquad  \frac{D}{dt}\gamma'(t)=
        0 , \label{eq:8}
      \end{equation}
      where $\frac{D}{dt} v$ denotes the covariant derivative  along the curve
      $\gamma(t)$ of the velocity vector field $v(t)=\gamma'(t)$ defined along the curve.
    \end{enumerate}
The equivalence of the above conditions is the content of the Hopf-Rinow theorem.
  \item If $M$ is an embedded closed submanifold of $\R^d$, then it is
    complete.
  \item\label{inj_radious} If $M$ is complete, for all $m\in M$ the exponential map is
    \[ \exp_m: T_m(M) \to M \qquad \exp_m(v)= \gamma(1), \]
     where $\gamma$
    is the geodesic defined by~\eqref{eq:8}.  There exists a maximal
    $r \in (0,+\infty]$ such that $\exp_m$ is a diffeomorphism from
    $B(0,r)\subset T_m(M)$ onto $B(m,r)\subset M$. The radius $r$ is called the injective
    radius at $m$ and it is denoted by $\operatorname{inj}(m)$.
  \item\label{exp} For any $m_0\in M$ and $0<r< \operatorname{inj}(m_0)$,
    the pair $( B(0,r),\exp_m)$ gives a local system of coordinates, which are called normal
    coordinates, on the open set $B(m_0,r) \subset M$. Fixing an orthonormal
    base $\set{e_i}_i^n$ of $T_{m_0}(M)$, the corresponding local chart is
    \begin{equation}
      \exp_{m_0}^{-1}(m) =\sum_{i=1}^n x^i(m) e_i \qquad m \in  B(m_0,r).\label{eq:26}
    \end{equation}
    Furthermore
    \[
      g= g_{ij} dx^idx^j \qquad g_{ij}= g(\partial_i,\partial_j)
    \]
    where $\partial_i=\frac{\partial}{\partial x_i}$ and, with slight
    abuse of notation, we denote by $g$ the $n\times n$ matrix
    $[g_{ij}]_{ij}$ and by $g^{ij}$ the elements of its inverse. The
    name ``normal'' refers to the property
    \begin{equation}
      \label{eq:30}
      g_{ij}(m_0)=\delta_{ij}  \qquad \partial_k g_{ij}(m_0)=0.
    \end{equation}
  \item The connection $\nabla$ in local coordinates is given by
    \begin{equation}
      \label{eq:32}
      (\nabla_Y X)^k= Y^i \partial_i( X^k) + Y^i X^j \Gamma^k_{ij}  
    \end{equation}
    where the Christoffel symbols  of second kind are
    \[
      \Gamma^k_{ij} =\frac{1}{2} g^{k\ell} \left( \partial_j g_{i\ell}
        + \partial_i g_{j\ell}- \partial_\ell g_{ii} \right) = g^{k
        \ell } \Gamma_{ij,\ell},
    \]
    see {\rm\cite[page 31]{pet16}}.
  \item\label{def_bounded}  The manifold $M$ has a bounded geometry if
    the two following 
    conditions hold true
    \begin{subequations}
      \begin{enumerate}[i)]
      \item there exists $r_M>0$ such that
        \begin{equation}
          \label{eq:18}
          \operatorname{inj}(m)\geq r_M \qquad m\in M;
        \end{equation}
      \item  given a local system of coordinates as in~\eqref{eq:26},
        for all multi-index $\alpha$
        \begin{equation}
          \label{eq:27}
  \operatorname{det}g(m) >c \qquad       \abs{D^\alpha g_{ij}(m)} \leq c_\alpha \qquad m\in B(m_0,r)
        \end{equation}
 where the constants $c, c_\alpha$  are uniform for all systems of
   normal coordinates, see~{\rm \cite[page 283]{trib92}}.
      \end{enumerate}
    \end{subequations}
  \end{enumerate}
\end{prop}

\begin{remark}
The definition  of ($C^\infty$)-bounded geometry is
given in~\cite[page 283]{trib92} and, under assumption~\eqref{eq:18},   it is
equivalent to assume that all  covariant derivatives of the Ricci
curvature tensor $R$ are bounded, see \cite[page 33]{CGT} and references
in~\cite[page 284]{trib92}. Furthermore,~\eqref{eq:27} is also equivalent to
assume that  for any multi-index $\alpha$ there exists a constant
$d_\alpha>0$ such that 
$\sup_{m\in B(m_0,r)} \{|D^{\alpha}\Gamma^i_{jk}(m)|\}\leq d_\alpha$ where
$d_\alpha$  are uniform for all systems of coordinates,
see~\cite[Proposition 2.4]{roe88}.  In \cite[Section 4.1]{fefupe16} there is a
weaker definition of manifold with bounded geometry, see
\cite[Lemma 2.6]{eldering13}. 
\end{remark}

\end{document}